\documentclass[12pt,reqno]{amsart}
\usepackage[margin=1in]{geometry}
\usepackage{graphicx}
\usepackage{amsmath,amsthm,amsfonts,mathrsfs,amssymb,float,color}
\usepackage{graphicx}
\usepackage{enumerate,enumitem}
\usepackage{textcomp}
\usepackage{verbatim}
\usepackage{ upgreek }
\usepackage[T1]{fontenc}
\usepackage[utf8]{inputenc}
\usepackage{hyperref}
\hypersetup{colorlinks=true,citecolor=red,linkcolor=blue}
\usepackage{epsfig,subfigure,fancybox,balance}
\usepackage{diagbox}
\usepackage{youngtab}
\usepackage[numbers]{natbib}
\usepackage{ dsfont }
\usepackage{enumitem}
\newcommand*{\rom}[1]{\expandafter\@slowromancap\romannumeral #1@}

\newcommand{\beq}[1]{\begin{equation} \label{#1}}
\newcommand{\eeq}{\end{equation}}
\newcommand{\bea}{\bed\begin{array}{rl}}
\newcommand{\eea}{\end{array}\eed}
\newcommand{\bed}{\begin{displaymath}}
\newcommand{\eed}{\end{displaymath}}
\newcommand{\barray}{\begin{array}{ll}}
\newcommand{\earray}{\end{array}}
\newcommand{\disp}{\displaystyle}
\newcommand{\ad}{&\!\disp}
\newcommand{\aad}{&\disp}
\newcommand{\beqa}[1]{\begin{equation}\label{#1}\barray}
\newcommand{\eeqa}{\earray\end{equation}}
\newcommand{\al}{\alpha}
\newcommand{\e}{\varepsilon}

\newcommand{\la}{\lambda}

\newcommand{\sg}{\sigma}
\newcommand{\Sg}{\Sigma}

\newcommand{\dl}{\delta}
\newcommand{\Dl}{\Delta}
\newcommand{\cd}{(\cdot)}

\newcommand{\sqe}{\sqrt{\e}}

\def\phi{\varphi}
\def\indi{{\bf 1}}
\def\half{\frac{1}{2}}

\newcommand{\CA}{{\mathcal A}}
\newcommand{\CF}{{\mathcal F}}

\newcommand{\CU}{\mathcal{U}}
\newcommand{\CX}{\mathcal{X}}

\newcommand{\CL}{\mathcal{L}}
\newcommand{\CP}{\mathcal{P}}
\newcommand{\CE}{\mathcal{E}}
\newcommand{\CK}{\mathcal{K}}

\newcommand{\CS}{\mathcal{S}}
\newcommand{\CV}{\mathcal{V}}

\newcommand{\CH}{\mathcal{H}}
\newcommand{\CN}{\mathcal{N}}

\newcommand{\EE}{{\mathbb E}}
\newcommand{\PP}{{\mathbb P}}

\newcommand{\NN}{{\mathbb N}}
\newcommand{\rr}{{\mathbb R}}

\newcommand{\TT}{{\mathbb T}}
\newcommand{\ZZ}{{\mathbb Z}}

\newcommand{\MM}{{\mathbb M}}

\newcommand{\lbar}{\overline}
\newcommand{\wdt}{\widetilde}
\newcommand{\wdh}{\widehat}

\newcommand{\qv}[1]{\langle #1 \rangle}
\newcommand{\bqv}[1]{\big\langle #1 \big\rangle}

\numberwithin{equation}{section}

\newtheorem{thm}{Theorem}[section]
\newtheorem{lem}[thm]{Lemma}

\newtheorem{rem}[thm]{Remark}

\newtheorem{ass}[thm]{Assumption}

\newcommand{\thmref}[1]{Theorem~{\rm \ref{#1}}}
\newcommand{\lemref}[1]{Lemma~{\rm \ref{#1}}}

\newcommand{\remref}[1]{Remark~{\rm \ref{#1}}}

\newcommand{\figref}[1]{Figure~{\rm \ref{#1}}}
\newcommand{\assmref}[1]{Assumption~{\rm \ref{#1}}}
\newcommand{\secref}[1]{Section~{\rm \ref{#1}}}

\newcommand{\appref}[1]{Appendix~{\rm \ref{#1}}}

\definecolor{gray}{rgb}{0.75, 0.75, 0.75}

\newcommand{\gronwall}{Gr\"{o}nwall }

\begin{document}
\title[SK approximation meets Khasminskii averaging principle]{
Smoluchowski-Kramers Approximation Meets Khasminskii Averaging Principles in Nonequilibrium Random Environments I
}
\author{Hongjiang Qian}
\address{Department of Mathematics and Statistics, Auburn University, Auburn, AL, 36832}
\email{hjqian.math@gmail.com}


\begin{abstract}
This work establishes a simultaneous Smoluchowski-Kramers approximation and Khasminskii averaging principle for a class of second-order stochastic differential equations (SDEs) in nonequilibrium random environments. The system describes the motion of a particle of mass $m>0$ subject to external forces, friction, and noise, all of which depend on a fluctuating environment  such as a stochastic heat bath. The environment is modeled by a fast-varying first-order SDE, where a parameter $0<\e\ll 1$ encodes the time-scale separation. Under the scaling $m=\e^2$, the slow process converges in probability to an effective dynamics with averaged drift and noise-induced coefficients. Our analysis utilizes a pathwise integration-by-parts formula and Poisson equations associated with the fast dynamics. Finally, numerical experiments are provided for demonstration.
\end{abstract}
\maketitle

\section{Introduction}
Let $\CX \subseteq \rr^{d_1}$ be a non-empty open set and $(\Omega, \CF,\{\CF_t\}, \PP)$ be a filtered probability space satisfying the usual condition. By Newton's second law, the motion of a particle with mass $m$ evolving in $\CX$, subject to external forces and a friction term proportional to its velocity, can be modeled by the following second-order stochastic differential equations (SDEs):
\beq{lang}
m \ddot{x}_t^{m} = b(x_t^m)- \la(x_t^m)\dot{x}_t^m + \sg(x_t^m)\dot{W}_t,\quad x_0^\e =x_0 \in \CX, \; \dot{x}^\e_0 = v_0 \in \rr^{d_1},
\eeq
where $b,\la$, and $\sg$ are suitable measurable functions. The $W_t$ is an $n_1$-dimensional standard Brownian motion. The above second-order SDE provides a stochastic model for a broad class of physical and engineering systems, including the motion of colloidal particles in fluids \cite{Nel20}, turbulence diffusion and stochastic accelerations model \cite{KP79,KP80}, and tracking problems such as camera-based  object tracking \cite{Pap10}. Moreover, many classical second-order differential equations arising in physics and engineering, such as Airy's equation in quantum mechanic and optics, the Duffing equation in nonlinear oscillations, or the Li\'{e}nard and Rayleigh equations in electrical circuits and control system, lead to second-order SDEs of the form \eqref{lang} when subjected to random perturbations.

A fundamental question associated with the second-order SDE \eqref{lang} concerns the small-mass limit: \textit{Can the second-order stochastic dynamics be effectively approximated by first-order stochastic models as $m \to 0$?} Under suitable assumptions, the answer is affirmative. The limiting equation is known as the \textit{Smoluchowski-Kramers approximation} (or \textit{small-mass limit}), dating back to the seminal works of Smoluchowski \cite{Smo16} and Kramers \cite{Kra40}. This approximation provides a rigorous link between the Newtonian and overdamped stochastic dynamics, describing the asymptotic transition from second-order to first-order stochastic differential equations as the particle mass vanishes. The small-mass limit captures the effective motion of microscopic particles in viscous environment and plays a central role in nonequilibrium statistical mechanics. We refer to \cite{CF05,FH11,Fre04,HHV16,HMVW15} and reference therein for comprehensive treatment and further developments.

In this paper, we investigate the Smoluchowski-Kramers approximation of such second-order stochastic systems in \textit{nonequilibrium random environment}, which is modeled by a fast-varying first-order SDE. More precisely, we study the  small-mass limit of the following \textit{fully-coupled slow-fast} stochastic system:
\beq{fast-slow}\left\{\barray
m \ddot{x}_t^m  & = b(x_t^m, y_t^\e) -\la(x_t^m, y_t^\e) \dot{x}_t^m+ \sg(x_t^m, y_t^\e)\dot{W_t}, \quad x_0^m = x_0\in \CX,\, \dot{x}_0^m=v_0 \in \rr^{d_1}, \\
\disp dy_t^\e & = \disp \frac{1}{\e} f(x_t^\e, y_t^\e)dt + \frac{1}{\sqe} g(x_t^\e, y_t^\e) d B_t,\quad y_0^\e= y_0\in \rr^{d_2}, 
\earray\right.\eeq
where $0<\e < 1$ is a small parameter encoding the time-scale separation.
Here $b: \CX \times \rr^{d_2} \to \rr^{d_1}$, $\la: \CX \times \rr^{d_2} \to \rr^{d_1\times d_1}$, and $\sg: \CX \times \rr^{d_2} \to \rr^{d_1 \times n_1}$ denote the force field, friction, and diffusion coefficient for the slow process $x_t^m$, respectively. The functions $f: \rr^{d_1} \times \rr^{d_2} \to \rr^{d_2}$ and $g: \rr^{d_1} \times \rr^{d_2} \to \rr^{d_2 \times n_2}$ are the drift and diffusion coefficients for the fast process $y_t^\e$. The processes $W_t$ and $B_t$ are independent standard Brownian motions of dimensions $n_1$ and $n_2$, respectively, defined on the probability space $(\Omega, \CF,\{\CF_t\}, \PP)$. Throughout the paper, we focus on $\CX=\rr^{d_1}$, in contrast to the bounded setting explored in \cite{HMVW15}.

For the investigation of particle dynamics in complex random media, it is essential to introduce a random environment variable $y_t^\e$ to encode microscopic fluctuations and thermodynamic heterogeneity. A representative example from statistical physics is the motion of Brownian particles in a fluid whose thermodynamic properties, such as temperature, viscosity, and chemical composition, fluctuate rapidly due to molecular-scale interactions. In such nonequilibrium settings, the environment evolves on a time scale much faster than that of the particle's position. Due to the presence of the scaling factor $1/\e$, the fast component $y_t^\e$ captures rapidly evolving environmental degrees of freedom and induces an intrinsic multiscale structure. This separation of time scales furnishes a natural setting for averaging and homogenization, thereby connecting microscopic stochastic fluctuations with effective macroscopic transport dynamics. As a concrete example, consider one dimensional Brownian particles in \textit{stochastic} heat baths with
$\la(x,y)=\kappa_B\TT(y)/D(x)$ and $\sg(x,y)= \sqrt{2}\kappa_B \TT(y)/\sqrt{D(x)}$ in \eqref{fast-slow}. Here $\kappa_B$ is the Boltzmann constant and $D(x)$ is a hydrodynamic diffusion coefficient due to effects of particle wall interactions. The $\TT(y)$ denotes the instantaneous temperature of the surrounding heat bath, which is essentially not assumed to be in thermal equilibrium. This formulation extends the classical Langevin dynamics in equilibrium thermal environment to nonequilibrium surroundings and thus leads to fast-slow second-order stochastic dynamics. A detailed discussion and numerical experiment are presented in \secref{sec:BM-shb}. For more practical applications, we refer to the examples in \cite{HHV16,HMVW15}.

The principal regime of interest in this work corresponds to the scaling
\beq{scale}
m=\e^2
\eeq
under which the inertial relaxation and the environmental fluctuations evolve on comparable asymptotic time scales. This coupling naturally links the small-mass limit with the Khasminskii averaging principle and eventually leads to averaged drift and noised-induced coefficients. Our goal is to understand how the fast-varying random environment influences the Smoluchowski-Kramers approximation when both effects occur simultaneously. While one may alternatively treat $m$ and $\e$ as independent parameters by first letting $m\to 0$ to obtain the small-mass limit and then letting $\e\to 0$ to obtain the averaging principle. The present work focuses on the joint limit $m=m(\e)\to 0$, with particular emphasis on the scaling \eqref{scale}. A related problem was studied in \cite{HLL26} under the setting $\sigma = g = 1$ and $\epsilon = m^2$, where strong convergence was established. In contrast, our system is fully coupled, and our analysis focuses on convergence in probability.

Heuristically, since $m=\e^2$ tends to zero faster than $\e$ as $\e\to 0$, the dynamics $x_t^m$ can be approximated by a first-order slow-fast stochastic system:
\beq{wdt-x}\left\{\barray
d\wdt{x}_t^\e = [\la^{-1}(\wdt{x}_t^\e, \wdt y_t^\e)b(\wdt{x}_t^\e, \wdt y_t^\e) + S(\wdt{x}_t^\e, \wdt y_t^\e)]dt + \la^{-1}(\wdt{x}_t^\e, y_t^\e)\sg(\wdt{x}_t^\e, y_t^\e) dW_t,\quad \wdt x_0^\e=x_0 \\
\disp d\wdt y^\e = \frac{1}{\e} f(\wdt x_t^\e, \wdt y_t^\e)dt + \frac{1}{\sqe} g(\wdt x_t^\e, \wdt y_t^\e) dB_t, \quad \wdt y_0^\e = y_0. 
\earray\right.\eeq
Here $\la^{-1}$ denotes the inverse of the matrix $\la$, whose existence  is assumed throughout the paper.  
We denote by $\wdt x^\e_t$ the solution of \eqref{wdt-x}, emphasizing its dependence on the fast process $\wdt y_t^\e$ and hence on the parameter $\e$. The above SDE can be viewed as the classical Smoluchowski-Kramers approximation in \cite{HMVW15} when the random environment $y_t^\e$ is frozen as a parameter. Importantly, $S(x,y)$ denotes the \textit{noise-induced drift} arising from state-dependent friction, defined by
\beq{S-xy}
S_i(x,y): = \frac{\partial}{\partial x_\ell}\big[\la^{-1}_{ij}(x,y)\big] J_{j\ell}(x,y),
\eeq
where $J$ is the solution to the Lyapunov equation $J \la^\top + \la J = \sg \sg^\top $
with $\sg^\top$ denotes the transpose of a matrix. In \eqref{S-xy}, we adopt the Einstein summation convention for repeated indices, and this convention will be used consistently throughout the paper.

Since the process $\wdt y_t^\e$ evolves on a fast time scale as $\e\to 0$, it is natural to expect that $\wdt x^\e$ \textit{converges weakly} to an averaged process $\bar x$, which satisfies the following  effective SDE:
\beq{barx}
d\bar x_t = \big[(\lbar{\la^{-1}b})(\bar x_t)+ \bar S(\bar x_t)\big]dt + (\lbar{\la^{-1}\sg})(\bar x_t)dW'_t, \quad \bar x_0 =x_0 \in \rr^{d_1},
\eeq
for some Brownian motion $W'$, where above averaged coefficients are defined as 
\beqa{bar-b-sg}
(\lbar{\la^{-1}b})(x) \ad : = \int_{\rr^{d_2}} (\la^{-1}b)(x,y) \mu^x(dy) \\
(\lbar{\la^{-1}\sg})(x) \ad :=\sqrt{\int_{\rr^{d_2}}[(\la^{-1}\sg)(\la^{-1}\sg)^\top] (x,y) \mu^x(dy)},
\eeqa
with $(\la^{-1}\ell)(x,y):=\la^{-1}(x,y)\ell(x,y)$ for $\ell=b,\sg$, and
\beq{bar-S}
\bar{S}_i(x):=\int_{\rr^{d_2}} \frac{\partial}{\partial x_\ell} \Big[\la^{-1}_{ij}(x,y)\Big] J_{j\ell}(x,y) \mu^x(dy).
\eeq
Here $\mu^x(dy)$ denotes the unique invariant measure of $y_t^x$ with frozen slow variable $x$: 
\beq{yt-x}
dy_t^x = f(x, y_t^x)dt + g(x, y_t^x) dB_t, \quad y_0^x = y_0 \in \rr^{d_2}.
\eeq

Our main contribution is the establishment of a simultaneous Smoluchowski--Kramers approximation and Khasminskii averaging principle, yielding a rigorous derivation of the effective small-mass limit for the multiscale system \eqref{fast-slow} under the critical scaling $m=\e^2$.

In this paper, we assume that $\la^{-1}\sg$ is independent of the fast variable $y$, a condition satisfied in many physical models, including the example in \secref{sec:BM-shb}. Under this assumption, the averaged diffusion coefficient $\lbar{\la^{-1}\sg}$ in \eqref{barx} reduces to $\la^{-1}\sg$ (i.e., no averaging with respect to $y$ is required), and weak convergence strengthens to convergence in probability. More precisely, we prove that the process $x_t^{m}$ converges in probability, in $C([0,T];\mathbb{R}^{d_1})$, to $\bar x$, the solution of the averaged equation \eqref{barx} with $\lbar{\la^{-1}\sg}$ replaced by $\la^{-1}\sg$, as $\e \to 0$; see \thmref{thm:SK-avg}. Notably, even under this structural simplification, the system \eqref{fast-slow} remains fully coupled. A key ingredient in our analysis is a novel representation formula for $x_t^m$, derived via integration-by-parts arguments combined with solutions to Poisson equations associated with the fast dynamics. We further exploit the Kurtz-Protter framework \cite{KP91} to establish the small-mass limit, while the Poisson equation approach is used to capture the influence of fast environmental fluctuations. For the general weak convergence to \eqref{barx} when the effective diffusion coefficient $\lambda^{-1}\sigma$ depends on the fast variable, we refer the reader to \cite{QY26}, where the corresponding martingale problem is analyzed in detail.

To the best of our knowledge, this work provides the first rigorous analysis of the small-mass limit in nonequilibrium random environments for fully coupled second-order fast-slow systems. It opens the door to the study of small-mass limits for a broader class of stochastic dynamics in random environments, including interacting particle systems with mean-field interactions and stochastic partial differential equations.

The remainder of the paper is organized as follows. \secref{sec:method} outlines the proof strategy, highlights the main technical challenges, and reviews related literature. \secref{sec:pre} introduces the notation, assumptions, preliminaries, and main results. \secref{sec:SK} establishes the Smoluchowski–Kramers approximation in nonequilibrium random environments. \secref{sec:exm-num} presents a motivating example from physical applications, along with detailed explanations and numerical illustrations. \secref{sec:rem} concludes with remarks and directions for future research. Technical proofs are deferred to \appref{app:lem-proof}.

\section{Methodology and Literature Review}\label{sec:method}
\subsection{Methodology of the Proof} This subsection presents methodologies for  proving the Smoluchowski-Kramers approximation in nonequilibrium random environments with an emphasis on main technical challenges.

Define $v_t^m =\dot{x}_t^m$ as the velocity. The slow-fast system \eqref{fast-slow} can be written as 
\beq{xvy}\left\{\barray
d x_t^m \ad = v_t^m dt  \\
d v_t^m \ad = \bigg\{ \frac{
b(x_t^m, y_t^\e)}{m} -\frac{\la(x_t^m, y_t^\e)}{m} v_t^m \bigg\} dt + \frac{\sg(x_t^m, y_t^\e)}{m} dW_t\\
dy_t^\e \ad = \frac{1}{\e} f(x_t^\e, y_t^\e)dt + \frac{1}{\sqe} g(x_t^\e, y_t^\e) d B_t.
\earray\right.\eeq 

To rigorously analyze the limit of the system \eqref{fast-slow} as $\e \to 0$, we adopt the framework of \cite{HMVW15}, relying on an integration-by-parts representation of $x_t^m$ in the form 
\bea\ad 
x_t^m = x_0 + \wdh \CU_t^m +\int_0^t F(x_s^m) d\CH_s^m,
\eea
where $\wdh \CU_t^m \to 0$ and $\CH_s^m \to \CH_s$ as $m \to 0$. In our case, the coupling with the fast process yields that both $\wdh \CU_t^m$ and $F(x_s^m)$ depend on $y^\e$, yielding the abstract  representation
\beq{xtm-FH}
x_t^m =x_0 + \wdh \CU_t^m(x_t^m, y_t^\e) + \int_0^t F(x_s^m, y_s^\e)d \CH_s^m.
\eeq 

We employ the stochastic integral convergence framework developed by Kurtz and Protter \cite{KP91} to study the convergence of \eqref{xtm-FH}. The fast-oscillatory nature of the integrand necessitates a modification of this approach. Drawing on homogenization results for stochastic integrals \cite{Lej02} and limit theorems for first-order multiscale SDEs \cite{PV01,PV03,PV05,RX21,RX21a}, we therefore introduce the following averaged integrand
\bea\ad 
\bar{F}(x)= \int_{\rr^{d_2}} F(x,y)\mu^x(dy), \forall\,  x\in \rr^{d_1}, y\in \rr^{d_2},
\eea
where $\mu^x$ is the invariant measure of process $y^x$ defined in \eqref{yt-x} and rewrite \eqref{xtm-FH} as 
\beq{xtm-barF}\barray
x_t^m \ad = x_0 + \underbrace{\wdh{\CU}_t^m(x_t^m, y_t^\e) + \int_0^t F(x_s^m, y_s^\e) -\bar F(x_s^m) d \CH_s^m}_{\text{new } \wdh \CU_t^m(x_t^m, y_t^\e)\ =:\  \CU_t^m(x_t^m, y_t^\e)} + \int_0^t \bar{F}(x_s^m) d\CH_s^m \\
\ad = x_0 + \CU_t^m(x_t^m, y_t^\e) + \int_0^t \bar F(x_s^m) d \CH_s^m.
\earray\eeq

For $\CH_s^m=s$, the first integral on the right-hand side of the first line of \eqref{xtm-barF} gives the deterministic integral 
\beq{Fs}
\int_0^t F(x_s^m, y_s^\e) -\bar F(x_s^m)ds.
\eeq
To handle \eqref{Fs}, we apply the Poisson equation method, rewriting the integral \eqref{Fs} in terms of the solution to an associated Poisson equation; see \secref{sec:Poi}.

One of the main analytical difficulties is that this rewriting generates a term that prevents \eqref{Fs} from vanishing in $L^2(\Omega; C([0,T]; \rr^{d_1}))$ as $\e \to 0$. More precisely, it introduces a contribution involving the time derivative of the slow variable $x$, with coefficient $\e=\sqrt{m}$, namely $\sqrt{m} v_t^m$, whose second moment is only bounded (see \lemref{lem:mv}) and thus does not yield sufficient decay. To overcome this issue, we work at the differential level and rewrite the deterministic integral in terms of $dx_t^m$, using the identity $v_t^m dt = dx_t^m$. By collecting all terms involving $dx_t^m$ from the right-hand side of \eqref{xtm-barF} and moving them to the left-hand side, we factor out the associated matrix coefficient of $dx_t^m$. This leads to a new representation of the solution $x_t^m$ obtained by inverting this matrix coefficient; see \eqref{xt-j}. To analyze the resulting expression, we employ a Neumann series expansion for the inverse matrix. This step is essential to show that the term $\CU_t^m(x_t^m, y_t^\e)$ in \eqref{xtm-barF} converges to zero in $L^2(\Omega; C([0,T]; \rr^{d_1}))$.

In light of the above observation and difficulty, the convergence of the stochastic integral
\beq{FW}
\int_0^t F(x_s^m, y_s^\e) - \bar F(x_s^m)\, dW_s, 
\eeq
with $\CH_s^m = W_s$, requires careful treatment. Under our structural assumption that $\lambda^{-1}\sigma$ is independent of the fast variable $y$, the stochastic integral \eqref{FW} does not involve any fast-scale fluctuations in the integrand, which substantially simplifies the analysis. In contrast, when the diffusion coefficient depends on the fast process, the averaging literature typically yields only weak convergence. In our setting, since $\lambda^{-1}\sigma$ acts as the effective diffusion coefficient and is independent of $y$, we are able to strengthen the result to convergence in probability.

Another major challenge is to establish moment bounds for $x^m$ in $C([0,T]; \rr^{d_1})$ when $\CX = \rr^{d_1}$, which is also crucial for proving tightness. To address this issue, existing works typically rely on a mild representation of $x^m$ obtained via the variation-of-constants formula. Specifically, for any $0 \leq s \leq t \leq T$ and $m, \e > 0$, we define
\beq{def-Am}
A^m(t,s) := \frac{1}{m} \int_s^t \la(x_s^m, y_s^\e)ds, \quad A^m(t):= A^m(t,0).
\eeq
From variants of constant formula to \eqref{xvy}, we have
\beq{vm} 
v_t^m = v_0 e^{-A^m(t)} +\frac{1}{m} \int_0^t e^{-A^m(t,s)} b(x_s^m, y_s^\e) ds + \frac{1}{m} e^{-A^m(t)} \int_0^t e^{A^m(s)}\sg(x_s^m, y_s^\e)dW_s.
\eeq
Taking integration from $0$ to $t$ for \eqref{vm}, we can obtain the mild representation form of $x_t^m$:
\beq{xm-mild}\barray
x_t^m \ad = x_0 + v_0 \int_0^t e^{-A^m(s)}ds + \frac{1}{m} \int_0^t \int_0^s e^{-A^m(s,r)} b(x_r^m, y_r^\e)drds  \\
\aad\quad + \frac{1}{m} \int_0^t \int_0^s e^{-A^m(s,r)} \sg(x_r^m,y_r^\e)dW_r ds.
\earray\eeq
In \eqref{vm}, the essential difficulty lies in following stochastic integral term:
\beq{eA-intW}
e^{-A^m(t)} \int_0^t e^{A^m(s)} \sg(x_s^m, y_s^\e)dW_s.
\eeq

It is important to emphasize that we \textit{cannot directly} move the factor $e^{-A^m(t)}$ inside the stochastic integral in \eqref{eA-intW} and interpret the resulting expression as an It\^{o} integral. The reason is that $e^{-A^m(t)}$ is \textit{not adapted}, which creates a fundamental obstacle to using the mild formulation for establishing moment bounds and tightness of $x_t^m$. At the same time, the presence of $e^{-A^m(t)}$ is crucial for compensating the large factor $e^{A^m(s)}$ in the integrand of \eqref{eA-intW}, which arises from the scaling $1/m = 1/\e^2 \to \infty$ as $\e \to 0$. Resolving this tension constitutes a central technical challenge.

In the special case $\sg(x,y)=\sg(x)$, Sandra Cerrai and Mark Freidlin \cite{CF15} interpret the stochastic integral \eqref{eA-intW} as a \textit{pathwise integral} for each $\omega \in \Omega$, and then apply an integration-by-parts argument to derive the required estimates. However, this approach breaks down in the fully coupled setting where $\sg(x,y)$ depends on the fast variable. In this case, the presence of the fast process necessitates the use of It\^{o}'s formula in the integration-by-parts procedure, which in turn produces a stochastic integral of the same type as \eqref{eA-intW}. This leads to a circular structure that prevents closing the estimates; see \cite{Qia26a}.

Finally, we point out that the arguments in \cite[p.~8]{WSW24} and \cite{SWLW24} for establishing moment boundedness and tightness of the mild solution has a gap, as they rely on the use of It\^{o} calculus in handling \eqref{eA-intW}; see also the recent discussion in \cite{Bla26}.

The other natural approach to address the \textit{non-adaptedness} of the integrand in \eqref{eA-intW} is to employ Malliavin calculus. Using Malliavin calculus together with integration-by-parts formula in \cite[Proposition 1.3.3]{Nua06}, one can rewrite \eqref{eA-intW} as 
\beq{eA-intW1}\barray\ad
e^{-A^m(t)} \int_0^t e^{A^m(s)} \sg(x_s^m, y_s^\e)dW_s \\
\aad = \int_0^t e^{-A^m(t,s)} \sg(x_s^m, y_s^\e)dW_s + \int_0^t  e^{A^m(s)} \sg(x_s^m, y_s^\e)D_s e^{-A^m(t)}ds =: V_1^m(t)+ V_2^m(t),
\earray\eeq
where $D_{\cdot}$ denotes the Malliavin derivative. This interpretation first appeared in \cite{Ngu22}, where the author studied the exponential tightness of \eqref{eA-intW} with $t=T$, $m=
\e^2$, and with a small coefficient $\sqe$ multiplying $\sg$. In this setting, the term  $V_1^m$ can be interpreted as a Skorokhod integral with anticipating integrand. Although we have newly observed that the $e^{-A^m(t,s)}$ in $V_1^m$ can be viewed as a \textit{random evolution system} in the sense of \cite[Definition 3.1, p. 158]{LN98} and wish to apply maximal inequality for Skorokhod integral, this approach does not appear to be sufficient to establish tightness of $x^m$. The main obstacle arises from the presence of the singular coefficients $1/m$ in the computation of Malliavin derivative $D_s e^{-A^m(t)}$  of $V_2^m$, which prevents us obtaining uniform moment bounds with repsect to $m$ or $\e$.

For these reasons, we first establish a general Smoluchowski--Kramers approximation for the fully coupled slow--fast stochastic system \eqref{fast-slow} under \assmref{ass:mom-bdd}, which guarantees moment boundedness of $x^m$. The condition \assmref{ass:mom-bdd} can be verified when the diffusion coefficient of the slow component is independent of the fast variable, i.e., $\sg(x,y)=\sg(x)$. In this setting, the stochastic integral in \eqref{eA-intW} admits a pathwise interpretation, which allows the factor $e^{-A^m(t)}$ to be moved inside the integrand. For the general case of $y$-dependent coefficients $\sg(x,y)$, verifying \assmref{ass:mom-bdd} or establishing tightness of $x^m$ is deferred to our subsequent work. In that setting, where $\la^{-1}\sg$ may depend on both $x$ and $y$, one expects to obtain weak convergence of $x_t^m$ to $\bar x$, the solution of \eqref{barx}.

\subsection{Literature Review}

The Smoluchowski-Kramers (SK) approximation provides a rigorous framework for approximating second-order stochastic differential equations with small mass by first-order SDEs. This reduction significantly simplifies the analysis and numerical simulation of complex systems. The fundamental idea dates back to the seminal works of Smoluchowski \cite{Smo16} and  Kramers \cite{Kra40}. 

Substantial progress has been made in establishing rigorous small-mass limit under various assumptions. In particular, Freidlin \cite{Fre04} proved convergence in probability of the small-mass limit for systems with constant friction $\la$ and position-dependent noise $\sg(x)$. This result was extended in \cite{CF05,FH11}, and further generalized to include state-dependent friction and noise in the work of Hottovy et al. \cite{HMVW15}, where a noise-induced drift appears in the limiting dynamics. More recent developments include results for unbounded coefficients \cite{HHV16}, system driven by fractional Brownian motion \cite{Son20}, and McKean-Vlasov (mean-field) equation with state-dependent friction \cite{SWLW24,SLD24}, and among others. 

Parallel research has also focused heavily on the Smoluchowski-Kramers approximation for stochastic partial differential equations, particularly stochastic damped wave equations. We refer to works  \cite{CF06,CF06a,Han24} for constant damping coefficients, and \cite{CFS17,CX22,CX23,CX24,CX25,CD25} for state-dependent friction coefficients.

To the best of our knowledge, no rigorous results exist for the Smoluchowski-Kramers approximation in nonequilibrium random environments concerning fully coupled slow-fast second-order SDEs. Although works of \cite{NY20,NY21,QY22,NY24} address Langevin dynamics in random environment, their analyses focus primarily on large and moderate deviation principles under at least one of the following assumptions: (i) the diffusion coefficient $\sg$ is absent in equation \eqref{fast-slow}; see \cite{NY24}; (ii) $\sg$ has small intensity of order $\sqe$; (iii) the fast process $y^\e$ and the state variable $x_t^m$ are not fully coupled. In the first part of our work, we address the small-mass limit in random environment under a general setting, without imposing assumptions (i)-(iii).

\section{Preliminaries, Assumptions, and Main Results}\label{sec:pre}

\subsection{Notation}
Let $T>0$ be a fixed terminal time. For a vector $a\in \rr^{d_1}$, $|a|$ denotes the Euclidean norm, and for a matrix $A\in \rr^{d_1\times d_1}$, $|A|$ denotes the induced operator norm. Throughout the paper, $C$ (with and without subscripts) denotes a generic positive constant independent of $m$ and $\e$, whose value may change from line to line.

For $\ell=b,\sg$, we use the notations $[\la^{-1}\ell](x,y)$ and $\ell_\la(x,y)$ interchangeably to denote the product of $\la^{-1}(x,y) \ell(x,y)$. Similarly, we denote the partial derivative of $\ell$ with respect to $x$ (or $y$) as $\nabla_x \ell$ ($\nabla_y \ell$) or $\ell_{x}$ ($\ell_y$), with the choice between them depending on convenience.

For a function $u\in C(\rr^{d_1})$, we define $|\cdot|_\infty := \sup_{x\in \rr^{d_1}} |u(x)|$. Let $\CS$ be a Polish space and denote by $\CP(\CS)$ the space of probability measures on $\CS$, equipped with the topology of weak convergence. The space of Borel measurable maps from $[0,T]$ into $\CS$ is denoted by $\MM([0,T]:\CS)$.

For function spaces, the subscript $b$ indicates boundedness in both variables $(x,y)$, while the subscript $p$ indicates polynomial growth in $y$-variable. In particular, for $Z\in L_p^\infty(\rr^{d_1}\times \rr^{d_2})$, there exists a constant $C, \wdh n>0$ such that
\bea
|Z(x,y)| \leq C(1+|y|^{\wdh n}),\quad \forall x\in \rr^{d_1}, y\in \rr^{n_2}.
\eea
For $0<\dl\leq 1$, we define $C_{p}^{0,\dl}:=C_p^{0,\dl}(\rr^{d_1}\times \rr^{d_2})$ as the space of functions that are locally H\"{o}lder continuous in $y$ and have at most polynomial growth in $y$. That is, there exists a constant $C, \wdh n>0$ such that for all $y_1, y_2\in \rr^{d_2}$, 
\bea
|Z(x,y_1)-Z(x,y_2)| \leq C(|y_1-y_2|^\dl \wedge 1) (1+|y_1|^{\wdh n}+|y_2|^{\wdh n}).
\eea
The associated semi-norm is given by 
\bea\ad
[Z]_{C_p^{0,\dl}}:=\sup_{x\in \rr^{d_1}}\sup_{|y_1|\leq 1,|y_2|\leq 1, |y_1-y_2|\leq 1} \frac{|Z(x,y_1)-Z(x,y_2)|}{|y_1-y_2|^\dl}.
\eea
For $0<\vartheta<1$, we define $C_p^{\vartheta,\dl}:= C_p^{\vartheta,\dl}(\rr^{d_1}\times \rr^{d_2})$ as  the space of all functions that are local $\vartheta$-local H\"{o}lder continuous in $x$, $\dl$-local H\"{o}lder continuous with polynomial growth in $y$. That is, for any $x_1,x_2 \in \rr^{d_1}, y_1, y_2 \in \rr^{d_2}$,
\bea
|Z(x_1, y_1)-Z(x_2, y_2)| \leq C \big[(|x_1-x_2|^\vartheta\wedge 1 )+(|y_1-y_2|^\dl \wedge 1)\big](1+|y_1|^{\wdh n}+|y_2|^{\wdh n}).
\eea
The corresponding quasi-norm for $C_p^{\vartheta,\dl}$ is defined as 
\bea\ad 
[Z]_{C_p^{\vartheta,\sg}}:=\sup_{|x_1-x_2|\leq 1} \sup_{|y_1|,|y_2|\leq 1, |y_1-y_2|\leq 1} \frac{|Z(x_1, y_1)- Z(x_2, y_2)|}{|x_1-x_2|^\vartheta+|y_1-y_2|^\dl}.
\eea
When $\vartheta,\dl\geq 1$, we use $C_p^{\vartheta,\dl}:=C_p^{\vartheta,\dl}(\rr^{d_1}\times \rr^{d_2})$ to denote the space of all functions $Z$ satisfying $\partial_x^{[\vartheta]} \partial_y^{[\dl]}Z \in C_p^{\vartheta-[\vartheta], \dl-[\dl]}$.

Throughout the paper, we use the notation $x_t^m$ and $v_t^m$ in place of $x_t^\e$ and $v_t^\e$ to emphasize the scaling $m=\e^2$, whose rationale is discussed in \remref{rem:me2}.

We assume the following moment boundedess of $x^m$.
\begin{ass}\label{ass:mom-bdd}\rm{ We assume there exists a constant $m_0>0$ (i.e. $\e_0>0$) such that for all $m\in (0,m_0)$ (i.e. $\e\in (0,\e_0)$), there exists a constant $C$ independent of $m$ such that 
\beq{mom-bdd}\sup_{t\in [0,T]}\EE |x_t^m|^k \leq C  \text{ and }
\EE\bigg[\sup_{t\in [0,T]}|x_t^m|^k \bigg] \leq C, \quad \forall\; k\in \NN.
\eeq
}
\end{ass}
In addition to \assmref{ass:mom-bdd}, we impose the following explicit conditions on coefficients of the slow-fast stochastic system \eqref{fast-slow}.

\begin{ass}\label{ass:gd-lip}\rm{
We assume
\begin{itemize}
\item[(A1)] The functions $b:\rr^{d_1} \times \rr^{d_2} \to \rr^{d_1}$, $\la: \rr^{d_1} \times \rr^{d_2} \to \rr^{d_1 \times d_1}$, and $\sg: \rr^{d_1} \times \rr^{d_2} \to \rr^{d_1 \times n_1}$ are continuously differentiable and globally Lipschitz continuous. More precisely, there exists a constant $C>0$ such that for all $x_1,x_2\in \rr^{d_1}$, $y_1, y_2\in \rr^{d_2}$ and for $\ell=b,\la$, and $\sg$,
\bea\ad 
|\ell(x_1,y_1)-\ell(x_2,y_2)| \leq C (|x_1-x_2|+ |y_1 -y_2|),
\eea
Moreover, $b$ satisfies a linear growth condition $
|b(x, y)| \leq C(1+|x|)$, uniformly in $y$.

\item[(A2)] The mapping $\la: \rr^{d_1} \times \rr^{d_2} \to \rr^{d_1}\times \rr^{d_1}$ belongs to $C_b^{1,2}(\rr^{d_1}\times \rr^{d_2})$ and has bounded partial derivatives. Moreover, there exists a constant $\la_0>0$ such that for all $\xi\in\rr^{d_1}$,
\bea\ad
\inf_{(x,y)\in \rr^{d_1}\times \rr^{d_2}}\qv{\la(x,y)\xi, \xi}_{\rr^{d_1}} \geq \la_0 |\xi|^2.
\eea

\item[(A3)] The matrix $\sg:\rr^{d_1} \times \rr^{d_2} \to \CL(\rr^{n_1}, \rr^{d_1})$ is invertible, bounded, and continuously differentiable with bounded partial derivatives.

\item[(B1)] The drift coefficient $f: \rr^{d_1} \times \rr^{d_2} \to \rr^{d_2}$ is bounded and continuous. Moreover,
\beq{prod-yf}
\lim_{|y|\to \infty} \sup_{x\in \rr^{d_1}}\qv{y,f(x,y)}=-\infty.
\eeq
\item[(B2)] The diffusion matrix $G:= gg^\top$ is non-degenerate in $y$ (uniformly with respect to $x$), bounded, and uniformly continuous. Besides, there exist constants $\beta_1, \beta_2>0$ such that
\bea\ad 
0 < \beta_1 \leq \frac{\qv{(gg^\top)(x,y)y, y}_{\rr^{d_2}}}{|y|^2} \leq \beta_2.
\eea
\end{itemize}
}
\end{ass}

\begin{rem}\label{rem:ass-cond}\rm{ We note that
\begin{itemize}
\item[(i)] Under assumptions (A1)--(A3) and (B1)--(B2), the slow-fast stochastic system \eqref{fast-slow} admits a  unique strong solution $(x^m, y^\e)\in C([0,T];\rr^{d_1}) \times C([0,T];\rr^{d_2})$.
\item[(ii)] The  conditions (B1) and (B2) guarantee the existence of a unique invariant measure $\mu^x(dy)$ for the frozen process $y_t^x$ defined in \eqref{yt-x}; see \cite{Ver97,Kha11} for details.
\item[(iii)] The assumption (A1)--(A3) imply that the function $\la^{-1}b$ is locally Lipschitz in $(x,y)$ and has linear growth in $x$, uniformly in $y$. More precisely, there exits a constant $C>0$ such that for any $x,x_1,x_2\in \rr^{d_1}$ and $y,y_1,y_2\in \rr^{d_2}$, we have 
\begin{equation}\label{lip-la-b}\begin{aligned}
|(\la^{-1}b)(x_1,y_1)-(\la^{-1}b)(x_2, y_2)|   \ad 
\leq C(1+|x_1|+|x_2|)\big(|x_1 -x_2|+ |y_1-y_2|\big) \\[-2pt]
|(\la^{-1}b)(x,y)| \ad \leq C(1+|x|)
\end{aligned}
\end{equation}
Similarly, we obtain that ${\la^{-1}\sg}$ is bounded and Lipschitz in $x$ and $y$. 

Moreover, the function $S(x,y)$ defined in \eqref{S-xy} is  uniformly bounded and globally Lipschitz in $(x,y)$. That is, there exists a constant $C>0$ such that 
\begin{equation}\label{lip-S}\begin{aligned}
|S(x_1,y_1)-S(x_2,y_2)| \ad  \leq C(|x_1-x_2|+|y_1-y_2|)\\
|S(x,y)| \ad \leq C.
\end{aligned}
\end{equation}
Consequently, the averaged limiting equation \eqref{barx} has a unique strong solution $\bar x$.

\item[(iv)] Define 
\beq{Z12}\barray
Z_1(x,y) :=(\la^{-1}b)(x,y)-(\lbar{\la^{-1}b})(x,y)\quad 
Z_2(x,y) :=S(x,y)-\lbar{S}(x)
\earray\eeq
where $\lbar{\la^{-1}b}$ and $\bar{S}$ are defined in \eqref{bar-b-sg} and \eqref{bar-S}, respectively. Then the functions $Z_1$ and $Z_2$ satisfy the Lipschitz continuity and growth conditions stated in \eqref{lip-la-b} and \eqref{lip-S}, respectively.
\end{itemize}
}
\end{rem}

The main result of this work is summarized below.

\begin{thm}\label{thm:SK-avg} Suppose that \assmref{ass:mom-bdd} and \assmref{ass:gd-lip} hold. Let $(x^m,y^\e)$ be the solution of slow-fast stochastic system \eqref{fast-slow}. Suppose that $\la^{-1}\sg$ is independent of the fast variable $y$. Under the scaling $m=\e^2$, as $\e\to 0$, the process $x^m$ converges in probability to $\bar x$, where $\bar x$ is the unique strong solution of the following averaged equation:
\bea
d\bar x_t = \big[(\lbar{\la^{-1}b})(\bar x_t)+ \bar S(\bar x_t) \big]dt + (\la^{-1}\sg)(\bar x_t)dW_t, \quad 
\bar x_0 =x_0,
\eea
with effective coefficients defined in \eqref{bar-b-sg} and \eqref{bar-S}.
\end{thm}
\begin{rem}\rm{
   We emphasize that the above result yields convergence in probability, due to the independence of $\lambda^{-1}\sigma$ from the fast variable $y$. When $\lambda^{-1}\sigma$ depends on $y$, one must instead analyze the associated martingale problem to establish weak convergence. This, together with the verification of \assmref{ass:mom-bdd}, is deferred to our subsequent  work.}
\end{rem}

\subsection{Preliminaries} In what follows, 
we recall essential tools that will be used in \secref{sec:SK} for  \thmref{thm:SK-avg} and present some preliminary results.
\subsubsection{Convergence of stochastic integrals}\label{sec:con-sint}
We recall the convergence of stochastic integrals below as our main idea for the study of Smoluchowski-Kramers approximation in nonequilibrium random environment in \secref{sec:SK}; see Kurtz and Protter \cite{KP91} and \cite{HMVW15} for details.

For $n\in \NN$, consider $(\CU^m, \CH^m)\in C([0,T]; \rr^{d_1} \times \rr^n)$ adapted to the filtration $\CF_t$, where $\CH_t^m$ is martingale with respect to $\CF_t$. Let $\CH_t^m = \CN_t^m + \CA_t^m$ be the Doob-Meyer decomposition. Let $F: \rr^{d_1} \to \rr^{d_1 \times n}$ be a continuous matrix-valued function, and suppose that $x^m \in C([0,T]; \rr^{d_1})$ satisfies the following integral equation:
\beq{xtm-fH}
x_t^m = x_0 + \CU_t^m + \int_0^t F(x_s^m) d\CH_s^m.
\eeq
Let $x\in C([0,T];\rr^{d_1})$ be the solution of the following equation:
\beq{xt-fH}
x_t =x_0 + \int_0^t F(x_s) d\CH_s.
\eeq

\begin{thm}[\rm{\cite[Theorem 5.10]{KP91}}]\label{thm:KP}
Assume that the following conditions hold:
\begin{itemize}
\item[(KP1)] (Convergence condition): $(\CU_t^m, \CH_t^m) \to (0,\CH)$ in probability with respect to the sup-norm in $C([0,T]; \rr^{d_1} \times \rr^n)$, i.e., for all $\dl>0$,
\beq{cond-1}
\PP\bigg[\sup_{s\in [0,T]}\Big(|\CU_s^m-0|+ |\CH_s^m - \CH_s|\Big) > \dl \bigg] \to 0, \quad \text{as } m \to 0.
\eeq
\item[(KP2)] (Tightness condition): the total variations $\{\CV_t(\CA^m)\}$ are stochastically bounded for each $t>0$, i.e., $\PP[\CV_t(\CA^m) > K] \to 0$ as $K \to \infty$, uniformly in $m$.
\end{itemize}
Suppose that there exists a unique global solution $x$ to \eqref{xt-fH}. Then as $m\to 0$, $x^m$ converges to $x$ in probability with respect to $C([0,T];\rr^{d_1})$.
\end{thm}

\subsubsection{Poisson equations}\label{sec:Poi}
The Poisson equation will play an important role to deal with the fast-varying random environment $y^\e$.  We refer to the work of \cite{PV01,PV03,PV05,RX21,SX25,Qia25,CDGOS26} and references therein on the development for Poisson equation and its applications in functional central limit theorem, large and/or moderate deviations, etc. Let us define
\beq{CL-x}
\CL_y^x := \sum_{i=1}^{d_2} f_i(x,y) \frac{\partial}{\partial y_i}+ \half \sum_{i,j=1}^{d_2} g_{i}(x,y)g_j(x,y) \frac{\partial^2}{\partial y_i \partial y_j},
\eeq
which is the generator of $y_t^x$ given by \eqref{yt-x}. Consider the following Poisson equation in $\rr^{d_2}$:
\beq{Poi-eq}
\CL_y^x u(x,y) = - Z(x, y)
\eeq
where the slow variable $x$ is regarded as a parameter. To guarantee the well-posedness of \eqref{Poi-eq}, we make the following ``centering condition'' of $Z$:
\beq{cen-cond}
\int_{\rr^{d_2}} Z(x, y)\mu^x(dy) =0,\quad \forall\, x\in \rr^{d_1},
\eeq
where $\mu^x(dy)$ is the invariant measure for $y_t^x$. 

The following result was proved in \cite[Theorem 2.1]{RX21} which will be used frequently. 
\begin{thm}\label{thm:Poi}
Suppose the coefficient $G:=gg^\top/2$ is non-degenerate in $y$ uniformly with respect to $x$ and the drift $f$ satisfies the weak recurrence condition 
\bea\ad
\lim_{|y| \to \infty} \sup_{x\in \rr^{d_1}} \bqv{y,f(x,y)}=-\infty.
\eea 
Assume that $f,g\in C_b^{\vartheta, \dl}$ with $0< \dl \leq 1$ and $\vartheta \geq 0$. Then for every $Z\in C_p^{\vartheta, \dl}$ satisfying \eqref{cen-cond}, there exists a unique solution $u\in C_p^{\vartheta,2+\dl}$ to Poisson equation \eqref{Poi-eq} satisfying \eqref{cen-cond} given by 
\bea\ad 
u(x, y)= \int_0^\infty \EE Z(x, y_t^x)dt.
\eea
Moreover, there exists a constant $\wdh n>0$ such that
\begin{itemize}
\item[(i)] (Case $\vartheta=0$) for any $x\in \rr^{d_1}$ and $y\in \rr^{d_2}$,
\beq{u-der-y}
|u(x, y)|+|\nabla_y u(x, y)|+ |\nabla_y^2 u(x, y)| \leq C_0 (1+|y|^{\wdh n}),
\eeq
where $C_0> 0$ depends only on $d_1, d_2$ and $\|f\|_{C_b^{0,\dl}}$, $\|g\|_{C_b^{0,\dl}}$ and $[Z]_{C_p^{0,\dl}}$;
\item[(ii)] (Case $\vartheta>0$) for any $y \in \rr^{d_2}$,
\bea
\|u(\cdot,y)\|_{C_b^{\vartheta}} \leq C_0 C_\vartheta (1+|y|^{\wdh n}),
\eea
where $C_\vartheta>0$ depends on $d_1, d_2$ and $\|f\|_{C_b^{\vartheta,\dl}}$, $\|g\|_{C_b^{\vartheta,\dl}}$ and $[Z]_{C_p^{\vartheta,\dl}}$.
\end{itemize}
\end{thm}
\begin{rem}\label{rem:trun-R}\rm{
A key assumption in \thmref{thm:Poi} is that, for each fixed $y \in \rr^{d_2}$, the function $Z(\cdot, y)$ is uniformly bounded with respect to the slow variable $x$. In \secref{sec:SK}, we apply \thmref{thm:Poi} to the functions $Z_1$ and $Z_2$ defined in \eqref{Z12}. However, $Z_1$ exhibits linear growth in $x$, uniformly in $y$, and is locally Lipschitz continuous in both $x$ and $y$ (see \eqref{lip-la-b}). This corresponds to the parameter regime $\vartheta = 1$, $\dl = 1$, and $\wdh n = 0$ in \thmref{thm:Poi}.

To satisfy the uniform boundedness requirement for $x$ in \thmref{thm:Poi}, we employ a truncation argument. For any fixed $R>0$, we define truncated versions of $b_\la$ and $\bar{b}_\la$ as
\beq{def-bla-R}
b_\la^R(x,y):= b_\la(x,y) \indi_{\{|x|\leq R\}} \text{ and } \bar{b}_\la^R(x)= \int_{\rr^{d_2}} b_\la^R(x,y)\mu^x(dy)
\eeq
Define $Z_1^R(x,y):=b_\la^R(x,y)-\bar{b}_\la^R(x)$.
  We would study the Poisson equation associated with $Z_1^R$. Consequently, the constant $C_0$ and $C_1$ ($C_\vartheta$ with $\vartheta=1$) in \thmref{thm:Poi} would depend on $R$ giving $C_0(R)$ and $C_1(R)$. We refer to the work \cite{CDGOS26}, which extends the analysis to allow for polynomial growth in both $x$ and $y$.}
\end{rem}

We apply It\^{o}'s formula to the solution of the Poisson equation evaluated at $(x_t^m, y_t^\e)$, which requires first-order derivatives of $u$ with respect to $x$ and second-order derivatives with respect to $y$. In view of \thmref{thm:Poi}, the required regularity in the fast variable $y$ is ensured by the uniform ellipticity of the operator $\CL_y^x$ in \eqref{CL-x}.

However, under our assumptions, $Z_1$ and $Z_1^R$ are only Lipschitz continuous in $x$. Consequently, \thmref{thm:Poi} yields at most Lipschitz regularity of $u$ in the $x$-variable, which is insufficient for a direct application of It\^{o}'s formula. To address this issue, we mollify $u$ following a standard procedure in the literature. Let $\varkappa : \rr^{d_1} \to [0,1]$ be a radial mollifier satisfying $\int_{\rr^{d_1}} \varkappa(x),dx = 1$, and such that for each $k \geq 1$ there exists a constant $C_k > 0$ with $|\nabla_x^k \varkappa(x)| \leq C_k \varkappa(x)$. For each $n \in \NN^*$, define
\bea
\varkappa_n(x):= n^{d_1} \varkappa(nx).
\eea
Given a function $u(x,y)$, we define the mollification of $u$ for $x$ variable as by 
\beq{mol-u}
u_n(x,y)= u(x,y) * \varkappa_n= \int_{\rr^{d_1}} u(x-x',y) \varkappa(x')dx'.
\eeq

We have the following estimates in \cite[Lemma 4.1]{RX21}.

\begin{lem}\label{lem:mol-est}
Let $u\in C_p^{\vartheta,\dl}$ with $0<\vartheta\leq 2, 0<\dl \leq 1$. Define $u_n$ by \eqref{mol-u}. Then we have 
\beq{u-un}
|u(\cdot, y)- u_n(\cdot, y)|_\infty \leq C_0 n^{-\vartheta} (1+|y|^{\wdh n})\text{ and } |\nabla_x^k u_n(\cdot, y)|_\infty \leq C_0 n^{k-\vartheta}(1+|y|^{\wdh n}),
\eeq
for $k=1,2$ and $C_0$ is a constant independent of $n$.
\end{lem}

\begin{rem}\rm{
Again, our setting implies $\vartheta=1,\dl=1,\wdh n=0$ and the associated constant in \lemref{lem:mol-est} depends on $R$, which is denoted by $C_0(R)$. For any fixed $R>0$, since $C_0(R)<\infty$, we have $C_0(R)\e\to 0$.}
\end{rem}

\section{Smoluchowski-Kramers Approximation in Random Environments} \label{sec:SK}
In this section, we prove \thmref{thm:SK-avg} by applying \thmref{thm:KP}. In \secref{sec:rep-int}, we derive a representation of the solution $x^m$ via an integration-by-parts argument. In \secref{sec:rep-poi}, we further develop a new representation of $x^m$ in the form \eqref{xtm-fH} using solutions to the associated Poisson equations. Several auxiliary lemmas are established in \secref{sec:aux-SK}. Finally, the proof of \thmref{thm:SK-avg} is completed in \secref{sec:proof-SK-avg}.

\subsection{First Representation of the Solution $x_t^m$}\label{sec:rep-int}
We rewrite the equation for $v_t^m$ as 
\bea
\la(x_t^m, y_t^\e) v_t^m dt = b(x_t^m, y_t^\e)dt + \sg(x_t^m, y_t^\e)dW_t - m dv_t^m.
\eea
By the assumptions that $\la(x,y)$ is invertiable in (A2), we obtain
\bea
dx_t^m = v_t^m dt = \la^{-1}(x_t^m, y_t^\e) b(x_t^m, y_t^\e)dt + \la^{-1}(x_t^m, y_t^\e)dW_t - m\la^{-1}(x_t^m, y_t^\e) dv_t^m.
\eea
Equivalently, the following integral form holds
\beq{xt-m}\barray
x_t^m \ad = x_0 + \int_0^t \la^{-1}(x_s^m, y_s^\e) b(x_s^m, y_s^\e)ds \\
\aad\quad + \int_0^t \la^{-1}(x_s^m, y_s^\e)\sg(x_s^m, y_s^\e) dW_s-\int_0^t m \la^{-1}(x_s^m, y_s^\e)dv_s^m.  
\earray\eeq

For the last term on the right-hand side of \eqref{xt-m}, applying the integration-by-parts formula followed by It\^{o}'s formula to $\lambda^{-1}(x,y)$ yields
\beq{m-la-v}\barray\ad
\int_0^t m \Big[(\la^{-1})_{ij}(x_s^m, y_s^\e)\Big] d(v_s^m)_j \\
\ad = (\la^{-1})_{ij}(x_t^m, y_t^\e) m (v_t^m)_j - (\la^{-1})_{ij}(x_0^m, y_0^\e) m (v_0^m)_j \\
\aad\quad - \int_0^t \frac{\partial}{\partial x_\ell}\Big[(\la^{-1})_{ij}(x_s^m, y_s^\e)\Big] m (v_s^m)_j (v_s^m)_\ell ds \\
\aad\quad - \int_0^t \frac{\partial}{\partial y_\ell} \Big[(\la^{-1})_{ij}(x_s^m, y_s^\e) \Big] m (v_s^m)_j \frac{1}{\e} f_\ell(x_s^m, y_s^\e) ds \\
\aad\quad -\half \int_0^t \frac{\partial}{\partial y_\ell \partial y_k} \Big[(\la^{-1})_{ij}(x_s^m, y_s^\e) \Big] m (v_s^m)_j \frac{1}{\e} g_\ell(x_s^m, y_s^\e) g_k (x_s^m, y_s^\e) ds \\
\aad\quad -\int_0^t \frac{\partial}{\partial y_\ell} \Big[(\la^{-1})_{ij}(x_s^m, y_s^\e) \Big] m (v_s^m)_j \frac{1}{\sqe} \big[g(x_s^m, y_s^\e) d B_s \big]_\ell.
\earray\eeq
To handle the term in the third line of the above equality, namely, 
\bea\ad
\int_0^t \frac{\partial}{\partial x_\ell} \Big[(\la^{-1})_{ij}(x_s^m, y_s^\e) \Big] m (v_s^m)_j (v_s^m)_\ell ds,
\eea
we follow \cite{HMVW15} and consider the product $m v_s^m (m v_s^m)^\top$. An application of It\^{o}'s formula gives
\beq{mv}\barray
d\Big[m v_s^m (mv_s^m)^\top \Big] \ad = \Big[m b(x_s^m, y_s^\e) (v_s^m)^\top - m \la(x_s^m, y_s^\e) v_s^m (v_s^m)^\top \Big]ds \\
\aad\quad + m (\sg(x_s^m, y_s^\e)dW_s) (v_s^m)^\top \\
\aad\quad + \Big[ mv_s^m b(x_s^m, y_s^\e)^\top - m v_s^m (v_s^m)^\top \la^\top(x_s^m, y_s^\e) \Big] ds \\
\aad\quad + mv_s^m \big(\sg(x_s^m, y_s^\e)dW_s \big)^\top + \sg(x_s^m, y_s^\e)\sg^\top(x_s^m, y_s^\e)ds.
\earray\eeq

Define 
\beq{wdt-U}
\wdt \CU_t^m :=\int_0^t m v_s^m b^\top(x_s^m, y_s^\e)ds + \int_0^t m v_s^m (\sg(x_s^m, y_s^\e)dW_s)^\top.
\eeq
We rewrite \eqref{mv} as 
\beq{mvv}\barray
\ad  m v_s^m (v_s^m)^\top ds \big[-\la^\top(x_s^m, y_s^\e) \big] + \big[- \la(x_s^m, y_s^\e) \big] m v_s^m (v_s^m)^\top ds \\
\ad = d \big[ m v_s^m (mv_s^m)^\top \big] - \sg(x_s^m, y_s^m) \sg^\top(x_s^m, y_s^m)ds - d\,\wdt \CU_s^m - d\,(\wdt \CU_s^m)^\top.
\earray\eeq
Denoting $m v_s^m (v_s^m)^\top ds$ by $V$, $-\la(x_s^m, y_s^\e)$ by $A$, and the right-hand side of equation \eqref{mvv} by $\Sg$, then \eqref{mvv} satisfies the so-called Lyapunov function
\beq{lya}
A V + V A^\top =\Sg.
\eeq
Under our assumptions, the real parts of the eigenvalues of $A$ are negative, \eqref{lya} has a unique solution given by
\beq{V}
V = - \int_0^\infty e^{Az} \Sg e^{A^\top z}dz.
\eeq
Then we obtain
\beq{mvv-1}\barray
m v_t^m (v_t^m)^\top dt \ad = -\int_0^\infty e^{-\la(x_t^m, y_t^\e) z } d\Big[m v_t^m (m v_t^m)^\top \Big] e^{-\la^\top(x_t^m, y_t^\e) z} dz \\
\aad\quad + \int_0^\infty e^{-\la(x_t^m, y_t^\e) z} \Big(\sg(x_t^m, y_t^\e)\sg^\top(x_t^m, y_t^\e) dt \Big) e^{-\la^\top(x_t^m, y_t^\e)z} dz \\
\aad\quad + \int_0^\infty e^{-\la(x_t^m, y_t^\e) z} \Big( d\, \wdt \CU_t^m + d\,(\wdt \CU_t^m)^\top \Big) e^{-\la^\top(x_t^m, y_t^\e)z}dz \\
\aad =:  d\,\Sg_t^1 + d\,\Sg_t^2 + d\,\Sg_t^3.
\earray\eeq
For the first term in the last line of \eqref{mvv-1}, we obtain
\bea\ad 
d (\Sg_t^1)_{ij} = -\int_0^\infty \Big(e^{-\la(x_t^m, y_t^\e)z} \Big)_{i k_1} \Big(e^{-\la^\top(x_t^m, y_t^\e)z} \Big)_{k_2 j} dz\, d \Big[ m (v_t^m)_{k_1} (m v_t^m)_{k_2}^\top\Big].
\eea
For the second term, we have $d\,\Sg_t^2 = J(x_t^m, y_t^\e) dt$, where $J: \rr^{d_1} \times \rr^{d_2} \to \rr^{d_1\times d_1}$ is the solution of the following Lyapunov equation 
\bea
J \la^\top + \la J = \sg \sg^\top.
\eea
This identity follows from differentiating the corresponding Lebesgue integrals in 
\bea\ad
\int_0^t \Big[J(x_s^m, y_s^\e) \la^\top(x_s^m, y_s^\e)+ \la(x_s^m, y_s^\e) J(x_s^m, y_s^\e) \Big]ds = \int_0^t \sg(x_s^m, y_s^\e) \sg^\top(x_s^m, y_s^\e) ds.
\eea
For the third term in the last line of \eqref{mvv-1}, using \eqref{wdt-U}, we express the entries of $\Sg_t^3$ as 
\bea
(\Sg_t^3)_{ij} \ad = \int_0^t \int_0^\infty \big(e^{-\la(x_s^m, y_s^\e)z} \big)_{i k_1} \Big\{ \big[m v_s^m b^\top(x_s^m, y_s^\e)\big]_{k_1 k_2} ds \\
\aad\quad + \big[m v_s^m \Big(\sg(x_s^m, y_s^\e)dW_s \Big)^\top \big]_{k_1 k_2} + \big[b(x_s^m, y_s^\e) (m v_s^m)^\top \big]_{k_1 k_2}ds \\
\aad\quad + \big[\sg(x_s^m, y_s^\e) dW_s (m v_s^m)^\top \big]_{k_1 k_2} \Big\}  \big( e^{-\la(x_s^m, y_s^\e)z} \big)_{k_2 j} dz \\

\aad =  \int_0^t \int_0^\infty \big(e^{-\la(x_s^m, y_s^\e)z}\big)_{i k_1} \big( e^{-\la^\top(x_s^m, y_s^\e)z}\big)_{k_2 j}dz \Big\{ \big[m v_s^m b^\top(x_s^m, y_s^\e)\big]_{k_1 k_2}ds  \\
\aad \quad + \big[m v_s^m (\sg(x_s^m, y_s^\e) dW_s)^\top\big]_{k_1 k_2} + \big[ b(x_s^m, y_s^\e) (m v_s^m)^\top \big]_{k_1 k_2} ds \\
\aad \quad + \big[\sg(x_s^m, y_s^\e)dW_s (m v_s^m)^\top\big]_{k_1 k_2} \Big\}.
\eea

Substituting \eqref{mvv-1} for $m v_t^m (v_t^m)^\top dt$ into \eqref{xt-m} and using \eqref{m-la-v}, we obtain
\beq{x-1}\barray
(x_t^m)_i \ad = (x_0)_i + (\wdh \CU_t^m)_i + \int_0^t \big(\la^{-1}(x_s^m, y_s^\e) b(x_s^m, y_s^\e)\big)_i ds \\
\aad\quad + \bigg(\int_0^t \la^{-1}(x_s^\e, y_s^\e) \sg (x_s^m, y_s^\e)dW_s \bigg)_i \\
\aad\quad + \int_0^t \frac{\partial}{\partial x_\ell}\Big[(\la^{-1})_{ij}(x_s^m, y_s^\e)\Big]  J_{j \ell}(x_s^m, y_s^\e) ds \\
\aad\quad + \int_0^t \frac{\partial}{\partial x_\ell} \Big[(\la^{-1}_{ij}(x_s^m, y_s^\e)) \Big]  \\
\aad \qquad\quad \times \bigg\{ \int_0^\infty \big(e^{-\la(x_s^m, y_s^\e)z}\big)_{j k_1} \big(e^{-\la^\top(x_s^m, y_s^\e)}\big)_{k_2 \ell} dz \bigg\}  d \big[(m v_s^m)_{k_1} (m v_s^m)_{k_2} \big] \\
\aad\quad + \int_0^t \frac{\partial}{\partial y_\ell} \Big[(\la^{-1})_{ij}(x_s^m, y_s^\e)\Big]  m (v_s^m)_j \frac{1}{\e} f_\ell(x_s^m, y_s^\e)ds \\
\aad \quad +\half \int_0^t \frac{\partial^2}{\partial y_\ell \partial y_k} \Big[
(\la^{-1})_{ij} (x_s^m, y_s^\e)\Big]  m (v_s^m)_j \frac{1}{\e} g_\ell (x_s^m, y_s^\e) g_k (x_s^m, y_s^\e) ds \\
\aad \quad + \int_0^t \frac{\partial}{\partial y_\ell} \Big[(\la^{-1})_{ij}(x_s^m, y_s^\e)\Big]  m (v_s^m)_j \frac{1}{\sqe} \big[g(x_s^m, y_s^\e)dB_s\big]_\ell\\
\aad =: (x_0)_i + (\wdh \CU_t^m)_i + I_1^\e(t) + I_2^\e(t) + I_3^\e(t) + I_4^\e(t) + I_5^\e(t) + I_6^\e(t) + I_7^\e(t),
\earray\eeq
where $\wdh \CU_t^m =\{(\wdh \CU_t^m)_i\}, i=1,\dots, d_1$ is defined by
\bea
(\wdh \CU_t^m)_i \ad := (\la^{-1})_{ij}(x_t^m, y_t^\e) m (v_t^m)_j - (\la^{-1})_{ij}(x, y) m (v_0^m)_j \\
\aad\quad + \int_0^t \frac{\partial}{\partial x_\ell} \Big[(\la^{-1})_{ij} (x_s^m, y_s^\e) \Big] \\
\aad\quad \times \bigg\{ \int_0^\infty \big(e^{-\la(x_s^m, y_s^\e) z}\big)_{j k_1} \big(e^{-\la^\top(x_s^m, y_s^\e) z}\big)_{k_2 \ell} dz \\
\aad\quad \times \Big(\big[m v_s^m b^\top(x_s^m, y_s^\e)\big]_{k_1 k_2}ds + \big[m v_s^m (\sg(x_s^m, y_s^\e)dW_s)^\top \big]_{k_1 k_2} \\
\aad\quad + \big[b(x_s^m, y_s^\e) (m v_s^m)^\top \big]_{k_1 k_2} ds + \big[\sg(x_s^m, y_s^\e) dW_s (m v_s^m)^\top \big]_{k_1 k_2} \Big) \bigg\}. 
\eea

\begin{rem}\label{rem:me2}\rm{ We note that
\begin{itemize}
\item[(i)] Compared with the representation of $x_t^m$ in \cite{HMVW15}, the expansion of \eqref{x-1} contains three additional terms, namely, $I_5^\e, I_6^\e, I_7^\e$, which arise from the  dependence of the fast  process $y^\e$ in $\la$. 
\item[(ii)] As in \cite[Lemma 3]{HMVW15}, we show that for each $t \in [0,T]$, $\sqrt{m}, v_t^m$ is uniformly bounded (uniformly in $m$ and $\e$), and that $m|v_t^m| \to 0$ in $L^2(\Omega)$ as $m \to 0$. Consequently, in order to ensure that $I_7^\e \to 0$ in $L^2(\Omega; C([0,T]; \rr))$, we require the existence of $\varrho > 0$ such that $\sqrt{m}/\sqrt{\e} = \e^{\varrho} \to 0$ as $\e \to 0$, which leads to the scaling $m = \e^{1+2\varrho}$.

For simplicity, we take $\varrho = \tfrac{1}{2}$, yielding $m = \e^2$. The resulting additional factor of $\sqrt{\e}$ ensures that $I_7^\e$ vanishes in the limit. Moreover, since $m |v_t^m| \to 0$ in $L^2(\Omega)$, it follows that $I_4^\e \to 0$ as $\e \to 0$.
\end{itemize}
}
\end{rem}

\begin{rem}\label{rem:I56}\rm{
    We also note that 
\begin{itemize}
\item[(i)]
For the terms $I_5^\e$ and $I_6^\e$, our scaling yields
\beq{ga-f}\barray\ad
\int_0^t \frac{\partial}{\partial y_\ell}\Big[(\la^{-1})_{ij}(x_s^m, y_s^\e) \Big]m  \frac{1}{\e} f_\ell(x_s^m, y_s^\e) (v_s^m)_j ds \\
\aad = \sqrt{m} \int_0^t \frac{\partial}{\partial y_\ell} \Big[(\la^{-1})_{ij}(x_s^m, y_s^\e) \Big] f_\ell(x_s^m, y_s^\e) d(x_s^m)_j,
\earray\eeq
and
\beq{ga-gg}\barray\ad
\half \int_0^t  \frac{\partial^2}{\partial y_\ell \partial y_k} \Big[(\la^{-1})_{ij}(x_s^m, y_s^\e) \Big] m  \frac{1}{\e} g_\ell (x_s^m, y_s^\e) g_k(x_s^m, y_s^\e) (v_s^m)_j ds \\
\aad = \half \sqrt{m} \int_0^t \frac{\partial^2}{\partial y_\ell \partial y_k} \Big[(\la^{-1})_{ij}(x_s^m, y_s^\e) \Big] g_\ell (x_s^m, y_s^\e) g_k(x_s^m, y_s^\e) d(x_s^m)_j.
\earray\eeq
\item[(ii)] The key observation is that neither $I_5^\e$ nor $I_6^\e$ converges to zero in $L^2(\Omega; C([0,T];\rr))$ as $\e\to 0$, since $\sqrt{m}v_t^m$ is only \textit{bounded} in $L^2(\Omega)$. This motivates rewriting $(v_s^m)_j ds$ as $d(x_s^m)_j$ and transferring the terms in the above identities to the left-hand side of \eqref{x-1} in differential form, thereby obtaining a new representation of the solution $x_t^m$. 
\end{itemize}
}
\end{rem}

\subsection{Second Representation of the Solution}\label{sec:rep-poi}
To derive a representation of $x_t^m$ in the form \eqref{xtm-fH}, we first address the fast component $y_t^\e$ appearing in the terms $I_1^\e$, $I_2^\e$, $I_3^\e$, and $I_4^\e$.

For the term $I_1^\e(t)$, for any fixed $R>0$, we rewrite it as 
\beq{la-b}\barray\ad
\int_0^t (\la^{-1}b)_i (x_s^m, y_s^\e) ds  \\
\aad = \int_0^t (\la^{-1}b)_i (x_s^m, y_s^\e)-(\lbar{\la^{-1}b})_i(x_s^m) ds + \int_0^t (\lbar{\la^{-1}b})_i (x_s^m)ds \\
\aad = \int_0^t \big[(\la^{-1}b)_i(x_s^m, y_s^\e)-(\lbar{\la^{-1}b})_i(x_s^m)\big] \indi_{\{|x_s^m| \leq R\}}ds \\
\aad\quad+ \int_0^t \big[(\la^{-1}b)_i(x_s^m, y_s^\e)-(\lbar{\la^{-1}b})_i(x_s^m)\big] \indi_{\{|x_s^m|> R\}} ds + \int_0^t (\lbar{\la^{-1}b})_i(x_s^m)ds \\
\aad=: I_{1,1}^\e(t)+ I_{1,2}^\e(t) + I_{1,3}^\e(t),
\earray\eeq
where $(\la^{-1}b)_i(x,y):= (\la^{-1}(x,y) b(x,y))_i$ and $\lbar{\la^{-1}}b$ is defined in \eqref{bar-b-sg}.

We introduce a truncation parameter $R$ to ensure that the integrand in $I_{1,1}^\e$ is uniformly bounded with respect to the slow variable $x$; see \remref{rem:trun-R}. The term $I_{1,1}^\e$ is then rewritten using the solution of an associated Poisson equation, while $I_{1,3}^\e$ contributes to the effective small-mass limit.

Let $\phi^R(x,y)=(\phi_1^R(x,y),\dots, \phi_{d_1}^R(x,y))$ be  the solution to the Poisson equation 
\beq{Poi-phi}
\CL_y^x \phi(x, y) = - \big[(\la^{-1}b)(x,y) - (\lbar{\la^{-1}b})(x)\big] \indi_{\{|x|\leq R\}},
\eeq
where the superscript $R$ indicates the dependence on the parameter $R$. By the definition of $(\lbar{\la^{-1}b})(x)$ in \eqref{bar-b-sg}, the right-hand side of \eqref{Poi-phi} satisfies the ``centering condition'' \eqref{cen-cond}. 

By \thmref{thm:Poi}, for each fixed $x\in \rr^{d_1}$ and $R>0$, there exists a unique solution $\phi^R(x,y)$ satisfying the regularity and boundedness properties stated therein. In particular, $\phi^R(x,y)$ is uniformly bounded in both $x$ and $y$ (with bounds depending on $R$, but independent of $m$ and $\e$), and is twice continuously differentiable with respect to $y$. 

To apply It\^{o}'s formula, we mollify $\phi^R$ as described in \secref{sec:Poi}, obtaining a smooth approximation $\phi^{R,n}$. Applying the It\^{o} formula to $\phi_i^{R,n}(x_s^m,y_s^\e)$ yields
\beq{Ito-phi}\barray
\phi_i^{R,n}(x_s^m, y_s^
\e)-\phi_i^{R,n}(x_0, y_0) \ad = \int_0^t \frac{\partial}{\partial x_j}\phi_i^{R,n}(x_s^m, y_s^\e) (v_s^m)_j ds + \frac{1}{\e} \int_0^t \CL_y^{x_s^m} \phi_i^{R,n} (x_s^m, y_s^\e) ds\\ 
\aad\quad + \frac{1}{\sqe} \int_0^t \frac{\partial}{\partial y_\ell} \phi^{R,n}(x_s^m, y_s^\e) [g(x_s^m, y_s^\e)dB_s]_\ell.
\earray\eeq
Multiplying both sides of above equation by $\e$ and using \eqref{Poi-phi}, we obtain 
\beq{ga-b}\barray\ad 
\int_0^t \big[ (\la^{-1}b)_i(x_s^m, y_s^\e) -(\lbar{\la^{-1}b})_i (x_s^m)\big] \indi_{\{|x_s^m| \leq R\}} ds \\
\aad = \e \int_0^t \frac{\partial}{\partial x_j}\phi_i^{R,n} (x_s^m, y_s^\e) d(x_s^m)_j + \sqe \int_0^t \frac{\partial}{\partial y_\ell} \phi_i^{R,n} (x_s^m, y_s^\e) [g(x_s^m, y_s^\e)dB_s]_\ell \\
\aad\quad -\e \big[\phi_i^{R,n}(x_s^m, y_s^\e)- \phi_i^{R,n}(x_0, y_0)\big]+ \int_0^t \CL_y^{x_s^m} \phi_i^R(x_s^m, y_s^\e)- \CL_y^{x_s^m} \phi_i^{R,n}(x_s^m, y_s^\e)ds.
\earray\eeq

In the above, we write $d(x_s^m)_j$ in place of $(v_s^m)_j\,ds$ for the first integral on the second line of \eqref{ga-b}, for the same reason explained in \remref{rem:I56}. Such a term does not vanish in $L^2(\Omega; C([0,T]; \rr))$. The underlying reason is that the discrepancy between a term and its average over the fast process is only of order $O(\e)$, whereas the decay of $|v_s^m|$ occurs at rate $O(m)=O(\e^2)$; see \lemref{lem:mv}. The remaining terms on the right-hand side of \eqref{ga-b} vanish in $L^2(\Omega)$ as $\e \to 0$ and $n \to \infty$, for any fixed $R>0$.

For $I_2^\e(t)$, since $\la^{-1}\sg$ is independent of the fast variable $y$, the term simplifies to 
\beq{ga-sg-W}\barray\ad 
\Big(\int_0^t (\la^{-1}\sg)(x_s^m, y_s^\e)dW_s \Big)_i  = \Big(\int_0^t [\la^{-1}\sg](x_s^m) dW_s \Big)_i.
\earray\eeq

For $I_3^\e(t)$, by the definition of $S$ and $\bar S$ in \eqref{S-xy} and \eqref{bar-S}, we have
\bea \disp
\int_0^t S_i(x_s^m, y_s^\e)ds \ad = 
\int_0^t S_i(x_s^m, y_s^\e) - \bar{S}_i(x_s^m) ds +\int_0 \bar{S}_i(x_s^m)ds =: I_{3,1}^\e(t) + I_{3,2}^\e(t).
\eea
The term $I_{3,2}^\e(t)$ contributes to the effective drift in the small-mass limit, while $I_{3,1}^\e(t)$ is treated in the same manner as $I_{1,1}^\e(t)$. Let $\psi(x,y) = (\psi_i(x,y))_{i=1}^{d_1}$ denote the solution to the following Poisson equation:
\beq{Poi-psi} 
\CL_y^x \psi(x, y)= - \big[ S(x, y) - \lbar{S}(x) \big].
\eeq
By \thmref{thm:Poi}, there exists a  unique solution $\psi$ to \eqref{Poi-psi}. We mollify $\psi$ and denote the resulting smooth approximation by $\psi^{n}$. Applying the It\^{o} formula to $\psi_i^{n}(x_s^m, y_s^\e)$ and rearranging the resulting terms, we obtain
\beq{ga-J}\barray\ad 
\int_0^t \big[S_i(x_s^m, y_s^\e) - \lbar{S}_i(x_s^m)\big] ds \\
\aad = \e \int_0^t \frac{\partial}{\partial x_j}\psi_i^{n} (x_s^m, y_s^\e) d(x_s^m)_j + \sqe \int_0^t \frac{\partial}{\partial y_\ell}\psi_i^{n} (x_s^m, y_s^\e) [g(x_s^m, y_s^\e) dB_s]_\ell \\
\aad \quad - \e \big[\psi_i^{n} (x_s^m, y_s^\e)- \psi_i^{n}(x_0, y_0) \big]+ \int_0^t \CL_y^{x_s^m} \psi_i (x_s^m, y_s^\e) - \CL_y^{x_s^m} \psi_i^{n}(x_s^m, y_s^\e)ds.
\earray\eeq 

Let us now consider the term $I_4^\e$. Define
\bea 
D_{k_1 k_2}^i (x_s^m, y_s^\e)\ad := \frac{\partial}{\partial x_\ell} \big[(\la^{-1}_{ij}(x_s^m, y_s^\e))\big] \bigg[\int_0^\infty \big(e^{-\la(x_s^m, y_s^\e)z}\big)_{j k_1} \big(e^{-\la^\top(x_s^m, y_s^\e)z}\big)_{k_2 \ell} dz \bigg], \\
\lbar{D}_{k_1 k_2}^i(x) \ad:= \int_{\rr^{d_2}} D_{k_1 k_2}^i(x,y)\mu^x(dy), \quad \forall x\in \rr^{d_1}, y\in \rr^{d_2}.
\eea
Therefore, we can rewrite $I_4^\e(t)$ as
\bea\ad\!\!\!\!\!\!
\int_0^t D_{k_1 k_2}^i(x_s^m, y_s^\e) d \big[(m v_s^m)_{k_1} (m v_s^m)_{k_2} \big] \\
\aad\!\!\!\!\!\! = \int_0^t \Big\{D_{k_1 k_2}^i(x_s^m, y_s^\e) -\lbar{D}_{k_1 k_2}^i(x_s^m) \Big\}d \big[(m v_s^m)_{k_1} (m v_s^m)_{k_2}\big] + \int_0^t \lbar{D}_{k_1 k_2}^i(x_s^m)d \big[(m v_s^m)_{k_1} (m v_s^m)_{k_2}\big] \\
\aad\!\!\!\!\!\! =: I_{4,1}^\e(t) + I_{4,2}^\e(t).
\eea
For $I_{4,1}^\e(t)$, we show that it converges to zero in $L^2(\Omega; C([0,T]; \rr))$, using the boundedness of $D_{k_1 k_2}^i$ and $\lbar{D}{k_1 k_2}^i$, together with the fact that $m|v_t^m| \to 0$ in $L^2(\Omega)$ for each $t \in [0,T]$. The term $I{4,2}^\e(t)$ contributes to the integral appearing in the second line of \eqref{xtm-barF}.

To derive a new representation of $x_t^m$, we observe that the coefficients multiplying the differentials $d(x_t^m)_i$, which arise from \eqref{ga-f}, \eqref{ga-gg}, the left-hand side of \eqref{x-1}, and the leading terms on the right-hand sides of \eqref{ga-b} and \eqref{ga-J}, can be combined when expressed in differential form.

To proceed, noting that $\e = \sqrt{m}$, we define the matrix $M_{R,n}(x_s^m, y_s^\e)$ by
\beq{M}\barray
[M_{R,n}]_{ij}(x_s^m, y_s^\e) \ad = \dl_{ij} - \sqrt{m}\frac{\partial}{\partial y_\ell} \big[(\la^{-1})_{ij}(x_s^m, y_s^\e)\big] f_\ell(x_s^m, y_s^\e) \\
\aad\quad - \half\sqrt{m} \frac{\partial^2}{\partial y_\ell \partial y_k} \big[(\la^{-1})_{ij}(x_s^m, y_s^\e)\big] g_\ell (x_s^m, y_s^\e) g_k(x_s^m, y_s^\e) \\
\aad\quad - \sqrt{m} \bigg[\frac{\partial}{\partial x_j} \phi_i^{R,n} (x_s^m, y_s^\e)+ \frac{\partial}{\partial x_j} \psi_i^n(x_s^m, y_s^\e) \bigg],
\earray\eeq
where $\dl_{ij}$ denotes the Kronecker delta function. Hence, for each $x\in \rr^{d_1}$ and $y\in \rr^{d_2}$,  $M_{R,n}(x,y)$ can be expressed in the matrix form
\beq{M-mat}\barray
M_{R,n}(x,y)\ad := I - \sqrt{m}\nabla_y \la^{-1}(x,y)\cdot f(x,y) - \half\sqrt{m}\text{Tr}\big[\nabla_y^2 \la^{-1}(x,y)(gg^\top)(x,y)\big]\\
\aad\quad - \sqrt{m} \nabla_x \phi^{R,n}(x,y) - \sqrt{m} \nabla_x \psi^{n}(x, y).
\earray\eeq

Define
\beq{Dl-xy}\barray
\Dl_{R,n}(x,y)
\ad := \nabla_y \la^{-1}(x,y) \cdot f(x,y) + \half \text{Tr}\big[\nabla_y^2 \la^{-1}(x,y) (gg^\top)(x,y)\big]\\
\aad\quad +\nabla_x \phi^{R,n}(x,y)+ \nabla_x \psi^{n}(x,y).
\earray\eeq
Then $M_{R,n}(x,y)=I-\sqrt{m} \Dl_{R,n}(x,y)$. By \lemref{lem:mol-est} and \assmref{ass:gd-lip}, there exists a constant $C_R>0$ (depending only on $R$) such that
\beq{uni-Dl}
\sup_{x\in \rr^{d_1},y\in \rr^{d_2}} |\Dl_{R,n}(x,y)| \leq C_R.
\eeq
Consequently, for any fixed $R>0$, we have $\sqrt{m}|\Dl_{R,n}(x,y)|\leq \e\, C_R\to 0$ as $\e \to 0$. It implies that the matrix $M_{R,n}(x_s^m, y_s^\e)$ is positive definite for sufficiently small $\e$. In particular, $M_{R,n}(x_s^m, y_s^\e)$ admits an inverse, which is denoted by $M_{R,n}^{-1}(x_s^m, y_s^\e)=([M_{R,n}^{-1}]_{ij}(x_s^m, y_s^\e))$.

By considering the differential forms of \eqref{x-1}, \eqref{ga-f}, \eqref{ga-gg}, \eqref{Ito-phi}, and \eqref{ga-J}, we transfer the terms in \eqref{ga-f} and  \eqref{ga-gg}, as well as the first terms on the right-hand sides of \eqref{Ito-phi} and \eqref{ga-J}, to the left of \eqref{x-1}. Applying the inverse of $M_{R,n}$ and integrating from $0$ to $t$ yield
\beq{xt-j} \barray
(x_t^m)_j \ad = (x_0)_j + (\wdh \CU_t^{M,m})_j + \int_0^t \big[M^{-1}_{R,n} \big]_{ij}(x_s^m, y_s^\e) \Big((\lbar{\la^{-1} b})(x_s^m)\Big)_i ds \\
\aad + \int_0^t \big[M^{-1}_{R,n}\big]_{ij} (x_s^m, y_s^\e) \Big((\la^{-1}\sg)(x_s^m) dW_s \Big)_i + \int_0^t \big[M^{-1}_{R,n} \big]_{ij}(x_s^m, y_s^\e) \lbar{S}_i (x_s^m) ds \\
\aad + \int_0^t  \big[M^{-1}_{R,n} \big]_{ij}(x_s^m, y_s^\e) \lbar{D}_{k_1 k_2}^i(x_s^m) d \big[(m v_s^m)_{k_1} (m v_s^m)_{k_2} \big],
\earray\eeq
for some stochastic process $\wdh \CU_t^{M,m}$.

By decomposing $[M_{R,n}^{-1}]_{i,j}(x,y)= [M_{R,n}^{-1}]_{i,j}(x,y)-\dl_{ij}+\dl_{ij}$ for any $x\in \rr^{d_1},y\in \rr^{d_2}$, the representation \eqref{xt-j} of $x_t^m$ can be further rewritten as  
\beq{xt-j-new}\barray
(x_t^m)_j \ad = (x_0)_j + (\CU_t^{M,m})_j + \int_0^t \dl_{ij} \Big((\lbar{\la^{-1}b})(x_s^m)\Big)_i ds  + \int_0^t \dl_{ij} \Big((\la^{-1}\sg)(x_s^m) dW_s\Big)_i \\
\aad\quad + \int_0^t \dl_{ij} \lbar{S}_i(x_s^m)ds  + \int_0^t \dl_{ij} \lbar{D}_{k_1 k_2}^i(x_s^m) d \big[(m v_s^m)_{k_1} (m v_s^m)_{k_2} \big],
\earray\eeq
where $\CU_t^{M,m}=(\CU_t^{M,m})_{j=1}^{d_1}$ is given by 
\begingroup
\allowdisplaybreaks
\begin{align}\label{U-Mm}
(\CU_t^{M,m})_j & = \int_0^t \big[M^{-1}_{R,n} \big]_{ij}(x_s^m, y_s^\e) d \big[(\la^{-1})_{i\iota}(x_s^m, y_s^\e) m(v_s^m)_\iota \big] \notag \\
& + \int_0^t \big[M^{-1}_{R,n} \big]_{ij}(x_s^m, y_s^\e) \frac{\partial}{\partial x_\ell} \big[(\la^{-1})_{i\iota}(x_s^m, y_s^\e)\big] \notag\\
& \qquad \times \bigg[\int_0^\infty \Big(e^{-\la(x_s^m, y_s^\e)z}\Big)_{\iota k_1} \Big(e^{-\la^\top(x_s^m, y_s^\e)z}\Big)_{k_2 \ell}dz \notag\\
& \qquad \times \Big(\big[m v_s^m b^\top(x_s^m, y_s^\e)\big]_{k_1 k_2}ds + \big[m v_s^m (\sg(x_s^m, y_s^\e)dW_s)^\top \big]_{k_1 k_2} \notag \\
& \qquad + \big[b(x_s^m, y_s^\e) (m v_s^m)^\top \big]_{k_1 k_2} ds + \big[\sg(x_s^m, y_s^\e) dW_s (m v_s^m)^\top \big]_{k_1 k_2} \Big) \bigg] \notag \\
& + \sqe \int_0^t \big[M^{-1}_{R,n} \big]_{ij}(x_s^m, y_s^\e) \frac{\partial}{\partial y_\ell} \big[(\la^{-1})_{i\iota} (x_s^m, y_s^\e)\big]  \e (v_s^m)_\iota [g(x_s^m, y_s^\e) dB_s]_\ell \notag \\
& + \int_0^t \big[M^{-1}_{R,n} \big]_{ij}(x_s^\e, y_s^\e)  \big[(\la^{-1}b)_i(x_s^m, y_s^\e) - (\lbar{\la^{-1}b})_i(x_s^m)\big] \indi_{\{|x_s^m|> R\}} ds \notag \\
& + \sqe \int_0^t \big[M^{-1}_{R,n} \big]_{ij}(x_s^\e, y_s^\e) \frac{\partial}{\partial y_\ell} \phi_i^{R,n}(x_s^m, y_s^\e) [g(x_s^m, y_s^\e)dB_s]_\ell  \\
& + \sqe \int_0^t \big[M^{-1}_{R,n} \big]_{ij}(x_s^m, y_s^\e) \frac{\partial}{\partial y_\ell} \psi_i^n (x_s^\e, y_s^\e)[g(x_s^m, y_s^\e)dB_s]_\ell \notag\\
& - \e \int_0^t \big[M^{-1}_{R,n} \big]_{ij}(x_s^m, y_s^
\e) d \big[\phi_i^{R,n}(x_s^m, y_s^\e) \big] \notag\\
& - \e \int_0^t \big[M^{-1}_{R,n} \big]_{ij}(x_s^m, y_s^\e) d \big[\psi_i^n (x_s^m, y_s^\e)\big] \notag  \\
& + \int_0^t \big[M^{-1}_{R,n} \big]_{ij}(x_s^\e, y_s^\e)  \big[ \CL_y^{x_s^m} \phi_i^R (x_s^m, y_s^\e) - \CL_y^{x_s^m} \phi_i^{R,n} (x_s^m, y_s^\e) \big] ds \notag\\
& + \int_0^t \big[M^{-1}_{R,n} \big]_{ij}(x_s^\e, y_s^\e) \big[\CL_y^{x_s^m} \psi_i(x_s^m, y_s^\e)- \CL_y \psi_i^n(x_s^m, y_s^\e)\big] ds \notag\\ 
& + \int_0^t \big[M^{-1}_{R,n} \big]_{ij}(x_s^m, y_s^\e) \Big( D_{k_1 k_2}^i(x_s^m, y_s^\e) - \lbar{D}_{k_1 k_2}^i(x_s^m)\Big) d \big[(m v_s^m)_{k_1} (m v_s^m)_{k_2} \big] \notag\\
& + \int_0^t \Big[\big[M^{-1}_{R,n} \big]_{ij}(x_s^m, y_s^\e) - \dl_{ij}\Big] \big((\lbar{\la^{-1}b})(x_s^m)\big)_i ds \notag\\
& + \int_0^t \Big[\big[M^{-1}_{R,n} \big]_{ij}(x_s^m, y_s^\e) - \dl_{ij}\Big] \Big((\la^{-1}\sg)(x_s^m) dW_s\Big)_i \notag\\
& + \int_0^t \Big[\big[M^{-1}_{R,n} \big]_{ij}(x_s^m, y_s^\e) - \dl_{ij}\Big] \lbar{S}_i(x_s^m) ds \notag \\
&  + \int_0^t \Big[\big[M^{-1}_{R,n} \big]_{ij}(x_s^m, y_s^\e) -\dl_{ij} \Big] \lbar{D}_{k_1 k_2}^i(x_s^m) d \big[(m v_s^m)_{k_1} ( mv_s^m)_{k_2} \big] \notag\\
& =: \sum_{\tau=1}^{15} U_{\tau}^\e(t). \notag
\end{align}
\endgroup
The terms $\wdh \CU_t^{M,m}, \CU_t^{M,m}$, and $U_\tau^\e$ all depend on the parameter $R$ and $n$. For brevity, we suppress this dependence throughout the text.

We will show that $\CU_t^{M,m}$ converges to $0$ in $L^k(\Omega; C([0,T]; \rr))$ for $k=1$ or $k=2$, and hence in probability. It suffices to verify that each term $U_\tau^\e(t)$, $\tau = 1,\dots,15$, converges to zero in $L^k(\Omega; C([0,T]; \rr))$ by first letting $\e \to 0$, then $n \to \infty$, and finally $R \to \infty$. The corresponding arguments are given in \lemref{lem:U-tau}--\lemref{lem:U-phi-psi} in \secref{sec:aux-SK}.

\subsection{Auxiliary Lemmas}\label{sec:aux-SK}

The following key lemma shows that for any $t\in [0,T]$, $\sqrt{m}|v_t^m|$ is bounded and $m|v_t^m| \to 0$ in $L^2(\Omega)$ as $m\to 0$.

In contrast to the proof of \cite[Lemma 3]{HMVW15}, we do not assume that $x_t^m$ is bounded. Thus, the coefficients $b$ and $\sg$, evaluated at $x_t^m$, may be unbounded. Therefore, the argument in \cite{HMVW15} is not directly applicable, and new arguments are required.

\begin{lem}\label{lem:mv}
Suppose that \assmref{ass:mom-bdd} and \assmref{ass:gd-lip} hold. For each $t\in [0,T]$, the quantity $\sqrt{m} |v_t^m|$ is bounded in $L^2(\Omega)$. More precisely, there exists a constant $C_{T,\la_0,|v_0|}>0$, independent of $m$, such that 
\beq{bdd-mvs}
\EE\Big[m|v_t^m|^2\Big] \leq C_{T,\la_0,|v_0|}.
\eeq
Consequently,
\beq{0-mvs}
\EE\Big[|mv_t^m|^2\Big] \leq C_{T,\la_0,|v_0|}m.
\eeq
and hence $m|v_t^m| \to 0$ in $L^2(\Omega)$, thus in probability. Moreover, we also have
\beq{4-mv}
\EE|mv_t^m|^4 \leq C_{T,\la_0,|v_0|}.
\eeq
\end{lem}

\begin{proof} 
Let $\Psi(v)= m|v|^2/2$ denote the kinetic energy. Applying the It\^{o} formula to $\Psi$ yields
\bea
d\Psi(v_t^m) \ad = \bigg\{ \frac{\text{Tr}(\sg(x_t^m, y_t^\e)\sg^\top(x_t^m, y_t^\e))}{2m}+ b^\top(x_t^m, y_t^\e) v_t^m - (\la(x_t^m,y_t^\e) v_t^m)^\top v_t^m \bigg\} dt \\
\aad\quad + (\sg(x_t^m, y_t^\e)v_t^mdW_t)^\top v_t^m.
\eea
By assumption (A2), we have
\bea
(\la(x_t^m, y_t^\e) v_t^m)^\top v_t^m\geq \la_0 |v_t^m|^2.
\eea
It follows that
\bea
d\Psi(v_t^m) \ad \leq \bigg\{\frac{\text{Tr}(\sg(x_t^m, y_t^\e)\sg^\top(x_t^m, y_t^\e))}{2m}+ b^\top(x_t^m, y_t^\e) v_t^m  - \frac{2\la_0}{m} \Psi(v_t^m) \bigg\} dt \\
\aad \quad + (\sg(x_t^m, y_t^\e) dW_t)^\top v_t^m.
\eea

Applying the Duhamel's principle and then taking expectations, we obtain
\beq{phi-v}\barray
\EE\Psi(v_t^m) \ad \leq e^{-\frac{2\la_0}{m}t}\, \EE \Psi(v_0^m) +\EE \int_0^t e^{-\frac{2\la_0}{m}(t-s)}  b^\top (x_s^m, y_s^\e) v_s^m ds \\
\aad\quad + \EE\int_0^t e^{-\frac{2\la_0}{m}(t-s)} \frac{\text{Tr}(\sg(x_s^m,y_s^\e)\sg^\top(x_s^m,y_s^\e))}{2m}ds.
\earray\eeq
For the second term on the right-hand side of \eqref{phi-v}, the elementary
inequality $a c\leq (a^2+c^2)/2$ for all $a,c\in \rr$ yields
\beq{bv}\barray\disp 
\EE\int_0^t e^{-\frac{2 \la_0}{m}(t-s)} b^\top (x_s^m, y_s^\e) v_s^m ds \ad \leq \EE\int_0^t |b(x_s^m, y_s^\e)|^2 ds + \EE \int_0^t e^{-\frac{4\la_0}{m}(t-s)} \frac{2}{m} \Psi(v_s^m) ds \\
\aad \leq C + \EE\int_0^t e^{-\frac{4\la_0}{m}(t-s)}\frac{2}{m} \Psi(v_s^m) ds,
\earray\eeq
where the last inequality follows from \assmref{ass:mom-bdd} and the linear growth of $b$ in (A1).

For the last term in \eqref{phi-v}, 
the boundedness of $\sg$ implies that 
\beq{Ephi-3}\barray\ad 
\EE\Big|\int_0^t e^{-\frac{2\la_0}{m}(t-s)}\frac{\text{Tr}\big[\sg(x_s^m, y_s^\e)\sg^\top(x_s^m, y_s^\e)\big] }{2m} ds \Big| \leq \frac{C}{2m}\int_0^t e^{-\frac{2\la_0}{m}(t-s)} ds\leq C_{T,\la_0}.
\earray\eeq

Consequently, combining the estimates \eqref{phi-v}, \eqref{bv}, and \eqref{Ephi-3} yields
\bea
\EE\Psi(v_t^m) \ad \leq \frac{m}{2}e^{-\frac{2\la_0}{m}t}|v_0|^2+C_{T,\la_0,|v_0|} + \int_0^t e^{-\frac{4\la_0}{m}(t-s)}\frac{2}{m} \EE\Psi(v_s^m)ds \\
\aad \leq C_{T,\la_0,|v_0|} + \int_0^t \frac{2}{m}e^{-\frac{2\la_0}{m}(t-s)} \EE \Psi(v_s^m)ds,
\eea
where we note that, in the last line above, $C_{T,\la_0,|v_0|}$ denotes a generic constant that may vary from line to line. By the generalized \gronwall  inequality in \thmref{thm:g-Gron} together with \lemref{lem:CH}, taking $\CK(t,s)=\CK(t-s)=2e^{-2\la_0(t-s)/m}/m$, we obtain
\bea
\EE\Psi(v_t^m) \ad \leq C_{T,\la_0,|v_0|} + \int_0^t \CH(t,s)  C_{T,\la_0,|v_0|} ds\\
\aad \leq C_{T,\la_0,|v_0|} +\int_0^t \CK(t-s) \exp \bigg(\int_0^{t-s} \CK(r)dr\bigg) C_{T,\la_0,|v_0|} ds\\
\aad \leq C_{T,\la_0,|v_0|}\Big(1+e^{2/\la_0}/\la_0\Big).
\eea
Thus, \eqref{bdd-mvs} follows, and consequently \eqref{0-mvs} holds.

For \eqref{4-mv}, we take $\Psi(v)=m^4|v|^4/4$ and repeat an analogous argument to that used in \eqref{bdd-mvs}. It yields the desired bound. This completes the proof of the lemma.
\end{proof}

To proceed, we provide estimates for the matrix $M_{R,n}^{-1}(x,y)-I$. 
\begin{lem}\label{lem:est-M}
Fix $R>0$. For each $x\in \rr^{d_1}$ and $y \in \rr^{d_2}$, let $M_{R,n}(x,y)$ and $\Dl_{R,n}(x,y)$ be defined in \eqref{M-mat} and \eqref{Dl-xy}, respectively. Then 
\beq{M-I}
|M_{R,n}^{-1}(x,y)-I| \leq \frac{\sqrt{m} |\Dl_{R,n}(x,y)|}{1-\sqrt{m}|\Dl_{R,n}(x,y)|}.
\eeq
Moreover, for any fixed $R>0$, the estimate
\beq{MRn-bdd}
|M_{R,n}^{-1}(x,y)|\leq 2,
\eeq
hold for all sufficiently small $\e$ (equivalently, for all sufficiently small m).
\end{lem}

\begin{proof}
By \eqref{uni-Dl}, for sufficiently small $\e$, we have $\sqrt{m}|\Dl_{R,n}(x,y)| <1$. Thus, $M_{R,n}$ is invertiable. The Neumann series expansion for the matrix inverse yields
\bea
M_{R,n}^{-1}(x,y)-I \ad = \big(I-\sqrt{m} \Dl_{R,n}(x,y)\big)^{-1} -I = \sum_{k=1} ^\infty \big(\sqrt{m}\Dl_{R,n}(x,y) \big)^k.
\eea
Using the submultiplicativity of the induced matrix norm, we obtain
\bea\ad
|M_{R,n}^{-1}(x,y)-I| \leq  \sum_{k=1}^\infty m^{k/2}|\Dl_{R,n}(x,y)|^k = \frac{\sqrt{m}|\Dl_{R,n}(x,y)|}{1-\sqrt{m}|\Dl_{R,n}(x,y)|},
\eea
which proves \eqref{M-I}. Consequently,
\bea
|M_{R,n}^{-1}(x,y)| \ad \leq |I|+ |M_{R,n}^{-1}(x,y)-I|  \leq 1+ \frac{\sqrt{m}|\Dl_{R,n}(x,y)|}{1-\sqrt{m}|\Dl_{R,n}(x,y)|} =\frac{1}{1-\sqrt{m}|\Dl_{R,n}(x,y)|}.
\eea
By \eqref{uni-Dl}, we have 
\bea\ad
|M_{R,n}^{-1}(x,y)| \leq \frac{1}{1-\sqrt{m} C_R}.
\eea
Hence, for any fixed $R>0$, since $\sqrt{m}C_R = \e C_R \to 0$ as $\e \to 0$, we have $1-\sqrt{m}C_R \geq \tfrac{1}{2}$ for all sufficiently small $\e$, which implies
\bea
|M_{R,n}^{-1}(x,y)| \leq 2.
\eea
This completes the proof.
\end{proof}

The following lemmas devote to establishing that each $\CU_\tau^\e$ in \eqref{U-Mm} converges to zero in the space $L^2(\Omega;C([0,T];\rr))$. Proofs of \lemref{lem:U-tau}, \lemref{lem:U-dl}, and \lemref{lem:U-phi-psi} are postponed to \appref{app:lem55}--\appref{app:lem57}, respectively.

\begin{lem}\label{lem:U-tau}
For each $\e>0$ and any fixed $R>0$, let $e_1, e_2: \rr^{d_1} \times \rr^{d_2} \to \rr$ be continuous functions satisfying $|e_1(x,y)| \leq C_{e_1}(1+|x|)$ and $|e_2(x,y)| \leq C_{e_2}$ for some constant $C_{e_1}, C_{e_2} > 0$.  Then,  for each $i,j, \iota =1,\dots, d_1$ and $\ell=1,\dots, d_2$, we have
\beq{inner-mv}
\lim_{\e \to 0} \EE \bigg[ \bigg( \sup_{t\in [0,T]} \Big|\int_0^t \big[M^{-1}_{R,n} \big]_{ij}(x_s^m, y_s^\e) e_1(x_s^m, y_s^\e) m (v_s^m)_i ds \Big|\bigg)^2 \bigg]=0,
\eeq
\beq{mv-W}
\lim_{\e \to 0} \EE \bigg[\bigg(\sup_{t \in [0,T]} \bigg| \int_0^t  \big[M^{-1}_{R,n} \big]_{ij}(x_s^m, y_s^\e) e_1(x_s^m, y_s^\e) m (v_s^m)_i d(W_s)_j \bigg| \bigg)^2 \bigg]=0,
\eeq
and 
\beq{ev-B}
\lim_{\e \to 0} \EE \bigg[\bigg(\sup_{t\in [0,T]} \bigg| \sqe \int_0^t \big[M^{-1}_{R,n} \big]_{ij}(x_s^m, y_s^\e) e_2(x_s^m, y_s^\e) \sqrt{m} (v_s^m)_\iota d(B_s)_\ell \bigg|\bigg)^2 \bigg]=0.
\eeq
\end{lem}
The proof of \lemref{lem:U-tau} is deferred to \appref{app:lem55}. By applying \lemref{lem:U-tau} with appropriate choices of the functions $e_1$ and $e_2$, we conclude that $U_2^\e$ and $U_3^\e$ converge to zero in $L^2(\Omega; C([0,T]; \rr))$.

For the term $U_{11}^\e$, we first expand $d\big[(m v_s^m){k_1}(m v_s^m){k_2}\big]$ using \eqref{mv}, rewriting it in a form analogous to $U_2^\e$, and then apply a similar argument as in \lemref{lem:U-tau}. This yields $U_{11}^\e \to 0$ in $L^2(\Omega; C([0,T]; \rr))$ as $\e \to 0$.

\begin{lem}\label{lem:M-e1-R}
Suppose that \assmref{ass:mom-bdd} holds. Let $e_1: \rr^{d_1} \times \rr^{d_2} \to \rr$ be a  continuous function such that $|e_1(x,y)| \leq C_{e_1}(1+|x|)$ for some constant $C_{e_1}$. Then for each $i,j=1,\dots, d_1$, we have
\beq{la-b-R}
\lim_{R \to \infty} \lim_{\e\to 0} \EE\bigg|\sup_{t\in [0,T]} \int_0^t [M_{R,n}^{-1}]_{ij}(x_s^m, y_s^\e) e_1(x_s^m,y_s^\e) \indi_{\{|x_s^m| > R\}} ds \bigg|^2=0.
\eeq
\end{lem}
\begin{proof}
By \lemref{lem:est-M} and H\"{o}lder inequality, we have 
\bea\ad
\EE\bigg[\sup_{t\in [0,T]} \int_0^t [M_{R,n}^{-1}]_{ij}(x_s^m, y_s^\e) e_1(x_s^m, y_s^\e) \indi_{\{|x_s^m|>R\}} ds\bigg]^2 \\
\aad \leq \frac{4 T C_{e_1}^2}{(1-\sqrt{m}C_R )^2}   \EE \int_0^T (1+|x_s^m|^2) \indi_{\{|x_s^m|>R\}}ds \\
\aad \leq \frac{4 T C_{e_1}^2}{R \big( 1-\sqrt{m}C_R \big)^2}  \bigg\{ \EE \bigg (1+\sup_{s\in [0,T]}|x_s^m|^4 \bigg) \bigg\}^{1/2} \bigg\{ \EE \bigg(1+\sup_{s\in [0,T]}|x_s^m|^2 \bigg) \bigg\}^{1/2}.
\eea
Letting $\e\to 0$ and then $R\to \infty$, the desired limit follows.
\end{proof}
Applying \lemref{lem:M-e1-R} with suitable choices of $e_1$ to $U_4^\e$, we conclude that $U_4^\e \to 0$ in $L^2(\Omega; C([0,T]; \rr))$.

For the terms $U_\tau^\e$, $\tau = 12,13,14,15$, involving $\big[M^{-1}{R,n}\big]{ij}(x_s^m, y_s^\e) - \dl_{ij}$, we derive the following estimates.

\begin{lem}\label{lem:U-dl}
Suppose that \assmref{ass:mom-bdd} and \assmref{ass:gd-lip} hold.  Let  $e_1: \rr^{d_1} \times \rr^{d_2} \to \rr$ be a function satisfying $|e_1(x,y)|\leq C_{e_1}(1+|x|)$ for some constant $C_{e_1}>0$. Then, for each $i,j=1,\dots d_1$, we obtain
\beq{M-I-ds}
\lim_{\e \to 0} \EE\Big|\sup_{t\in [0,T]} \int_0^t \big( [M_{R,n}^{-1}]_{ij}(x_s^m, y_s^\e)-\dl_{ij}\big) e_1(x_s^m, y_s^\e) ds \Big|^2 =0
\eeq
and
\beq{M-I-dW} 
\lim_{\e \to 0} \EE\Big|\sup_{t\in [0,T]} \int_0^t \big( [M_{R,n}^{-1}]_{ij}(x_s^m, y_s^\e)-\dl_{ij}\big) e_1(x_s^m, y_s^\e) dW_s^i \Big|^2 =0
\eeq 
\end{lem}
The proof of \lemref{lem:U-dl}  is postponed to \appref{app:lem56}. Applying \lemref{lem:U-dl} with appropriate choices of $e_1$, we deduce that $U_\tau^\e \to 0$ in $L^2(\Omega; C([0,T]; \rr))$ for $\tau = 12,13,14$. For $U_{15}^\e$, we expand $d\big[(m v_s^m){k_1}(m v_s^m){k_2}\big]$ using \eqref{mv}, and then apply \lemref{lem:U-dl} to obtain $U_{15}^\e \to 0$ in $L^2(\Omega; C([0,T]; \rr))$ as $\e \to 0$.

We next turn to the terms involving $\phi_i^{R,n}$ and $\psi_i^n$, namely $U_\tau^\e$ for $\tau = 5,6,7,8,9,10$.

\begin{lem}\label{lem:U-phi-psi}
Suppose that \assmref{ass:mom-bdd} holds. Fix $R>0$. Let $\phi_i^R(x,y)$ and $\psi_i(x,y)$ denote the solutions of Poisson equations \eqref{Poi-phi} and \eqref{Poi-psi}, respectively. Their mollifications are denoted by $\phi_i^{R,n}(x,y)$ and $\psi_i^n(x,y)$. Then, for each $i,j=1,\dots,d_1$ and $\ell =1,\dots, d_2$
\beq{e3n-B}
\lim_{\e \to 0}
\EE\bigg| \sup_{t\in [0,T]} \sqe \int_0^t \big[M^{-1}_{R,n} \big]_{ij}(x_s^m, y_s^\e)\frac{\partial}{\partial y_\ell} \phi_{i}^{R,n}(x_s^m, y_s^\e) \big[g(x_s^m, y_s^\e)dB_s \big]_\ell \bigg|^2 =0,
\eeq
\beq{d-e3n}
\lim_{\e \to 0} \EE \bigg|\sup_{t\in [0,T]} \e \int_0^t \big[M^{-1}_{R,n} \big]_{ij}( x_s^m, y_s^\e) d \big[\phi_{i}^{R,n}(x_s^m, y_s^\e) \big] \bigg|^2 =0,
\eeq
and
\beq{e3n-CL}
\lim_{n \to \infty} \sup_{\e>0}\EE\bigg| \sup_{t\in [0,T]} \int_0^t \CL_y^{x_s^m} \phi_{i}^R(x_s^m, y_s^\e) -\CL_y^{x_s^m} \phi_{i}^{R,n}(x_s^m, y_s^\e)ds \bigg|^2 =0.
\eeq
Moreover, the above results remain valid if $\phi_i^{R,n}$ is replaced by $\psi_i^n$.
\end{lem}
\begin{proof}
The proof is postponed to \appref{app:lem57}.
\end{proof}

For the term $U_1^\e$, we can decompose it as
\bea\ad
U_1^\e = \int_0^t [M_{R,n}^{-1}]_{ij}(x_s^m, y_s^\e) d \big[(\la^{-1})_{i\iota}(x_s^m, y_s^\e) m(v_s^m)_\iota \big] =U_{1,1}^\e(t)+U_{1,2}^\e(t),
\eea
where
\bea 
U_{1,1}^\e(t)\ad := \int_0^t [M_{R,n}^{-1}]_{ij}(x_s^m, y_s^\e) m(v_s^m)_\iota d[(\la^{-1})_{i\iota}(x_s^m, y_s^\e)],\\
 U_{1,2}^\e(t)\ad := \int_0^t [(M_{R,n}^{-1})]_{ij}(x_s^m, y_s^\e) (\la^{-1})_{i\iota}(x_s^m, y_s^\e)d(mv_s^m)_\iota.
\eea
Using arguments analogous to those in \lemref{lem:U-tau}--\lemref{lem:U-phi-psi}, we can show that $U_{1,1}^\e$ and $U_{1,2}^\e$ both converge to zero in $L^2(\Omega; C([0,T];\rr))$ as $\e\to 0$. The details are omitted.

\subsection{Proof of \thmref{thm:SK-avg}} \label{sec:proof-SK-avg}
To prove \thmref{thm:SK-avg}, it remains to verify conditions of  \thmref{thm:KP}. We rewrite \eqref{xt-j-new} as the following compact form
\beq{xt-comp}
x_t^m = x_0 + \CU_t^{M,m} + \int_0^t \bar{F}(x_t^m) d\CH_t^m.
\eeq
where $\bar F: \rr^{d_1} \to \rr^{d_1\times (1+n_1+1+d_1^2)}$ is defined by
\bea
\bar F(x):= \big((\lbar{\la^{-1}b})(x), (\la^{-1}\sg)(x), \lbar{S}(x), \bar{F}^1(x), \dots, \bar F^{d_1}(x) \big).
\eea
Here $\lbar{S}: \rr^{d_1} \to \rr^{d_1}$ is given component-wise by 
\bea\ad 
\lbar{S}_i(x) = \int_{\rr^{d_2}} \frac{\partial}{\partial x_\ell} \big[(\la_{ij}^{-1}(x,y))\big] J_{j\ell}(x,y) \mu^x(dy),
\eea
and the components $\bar F^i(x): \rr^{d_1} \to \rr^{d_1 \times d_1}$ are defined by
\bea\ad 
\bar F^i_{k_1 k_2}(x):= \int_{\rr^{d_2}} \frac{\partial}{\partial x_\ell} \big[\la_{ij}^{-1}(x,y)\big] \bigg\{ -\int_0^\infty \Big(e^{-\la(x,y)z}\Big)_{j k_1} \Big(e^{-\la^\top(x,y)z}\Big)_{k_2 \ell} dz \bigg\} \mu^x(dy),
\eea
for $k_1,k_2 =1,\dots d_1$. 

Moreover, $\CH_t^m$ and $\CH_t$, with paths in $C([0,T];\rr^{1+n_1+1+d_1^2})$, are defined as  
\bea
\CH_t^m=\begin{bmatrix}t\\W_t\\ t\\ (m v_t^m)_1 (m v_t^m)- (m v_0)_1 (m v_0) \\ \vdots \\ (m v_t^m)_{d_1} (m v_t^m)- (m v_0)_d (m v_0)
\end{bmatrix},
\quad
\CH_t = \begin{bmatrix}t \\W_t\\ t\\ 0\\ \vdots \\ 0
\end{bmatrix}.
\eea
By \lemref{lem:mv}, we have $\CH^m \to \CH$ in probability in $C([0,T]; \rr^{1+n_1+1+d_1^2})$ as $\e \to 0$. Hence, we obtain $(\CU^{M,m}, \CH^m) \to (0,\CH)$ in probability in $C([0,T]; \rr^{d_1} \times \rr^{1+n_1+1+d_1^2})$. Consequently, condition $\text{(KP1)}$ in \thmref{thm:KP} is satisfied. 

We now verify the condition (KP2) in \thmref{thm:KP}. We seek the Doob-Meyer decomposition of $\CH_t^m$ and show that the total variations are stochastically bound, uniformly in $m$. Denote by $\CA_t^m$ the bounded-variation part. Using the expression $d\,[m v_s^m (mv_s^m)^\top]$ in \eqref{mv} and noting that the stochastic integrals are local martingales, we obtain
\bea
\CA_t^m = \begin{bmatrix}
t \\ 0 \\ t\\ (\wdt \CA_t^m)_1 \\ \vdots\\ (\wdt\CA_t^m)_{d_1}
\end{bmatrix}
\eea
where 
\beq{wdt-Am}\barray\ad
((\wdt \CA_t^m)_1, \dots, (\wdt \CA_t^m)_{d_1})\\
\aad = \int_0^t b(x_s^m, y_s^\e) m(v_s^m)^\top ds + \int_0^t m v_s^m b(x_s^m, y_s^m)^\top ds - \int_0^t \la(x_s^m, y_s^\e) mv_s^m (v_s^m)^\top ds \\
\aad  - \int_0^t m v_s^m (\la(x_s^m, y_s^\e) v_s^m)^\top ds + \int_0^t \sg(x_s^m, y_s^\e) \sg^\top(x_s^m, y_s^\e) ds.
\earray\eeq

It remains to show that $\CA_t^m$ is stochastically bounded. For the first two terms, H\"{o}lder's inequality and the fact that $mv_s^m \to 0$ in $L^2(\Omega)$ yield that, as $\e \to 0$,
\bea
\EE|b(x_s^m, y_s^\e) mv_s^m|
\ad \leq \Big(\EE|b(x_s^m, y_s^\e)|^2\Big)^{1/2}\Big(\EE|mv_s^m|^2\Big)^{1/2} \\
\ad \leq C \bigg(1+\EE\sup_{s\in [0,T]} |x_s^m|^2 \bigg)^{1/2} \Big(\EE|m v_s^m|^2\Big)^{1/2} \to 0.
\eea
Hence, the first two terms in \eqref{wdt-Am} vanish in probability. For the remaining terms, it suffices to show that 
$\EE|\sg(x_s^m, y_s^\e) \sg^\top(x_s^m, y_s^\e)|$ and $\EE|m (v_s^m)_i (v_s^m)^\top|$ are bounded uniformly in $m$ for $i=1,\dots, d_1$. By the boundedness of $\sg$ and \lemref{lem:mv}, we have 
\bea
\EE|\sg(x_s^m, y_s^\e) \sg^\top(x_s^m, y_s^\e)| \ad \leq C \text{ and }
\EE |m (v_s^m)_i v_s^m|  \leq \EE |m v_s^m|^2 \leq C.
\eea
Therefore, by the Chebyshev inequality, $\{V_t(\CA^m)\}$ is stochastically bounded, and $\CH_t^m$ satisfies condition (KP2). Consequently, we proved $x_t^m \to \bar x_t$ in probability as $\e \to 0$. This completes the proof of \thmref{thm:SK-avg}.

\section{Example: Brownian Particle in a Stochastic Heat Bath}\label{sec:exm-num}
\subsection{Brownian Particle in a Stochastic Heat Bath}\label{sec:BM-shb}
We consider a colloidal particle with small but nonzero mass $m$ suspended in a cylinder filled with a viscous fluid. The surrounding fluid acts as a heat bath, exerting both dissipative (frictional) and fluctuating (noisy) forces on the particle. The particle dynamics are given  by
\beq{exm}
\left\{\barray
dx_t^m \ad = v_t^m dt \\
dv_t^m \ad = \bigg\{ \frac{F(x_t^m)}{m}-\frac{\kappa_B \TT(y_t^\e)}{m D(x_t^m)} v_t^m  \bigg\} dt + \frac{\sqrt{2} \kappa_B \TT(y_t^\e)}{m \sqrt{D(x_t^m)}} dW_t \\
dy_t^\e \ad = \frac{1}{\e}f(x_t^m, y_t^\e)dt + \frac{\sg}{\sqe} dB_t,
\earray\right.
\eeq
where $m=\e^2$ denotes the small mass, $\e$ is the time scale separation parameter, 
$\kappa_B$ is the Boltzmann constant, and $\TT\cd$ represents the temperature of the stochastic heat bath.

In prior works such as \cite[Section 4.1]{HMVW15} and \cite{VHBWB10,BVHWB11}, the temperature $\TT\cd$ is assumed to be constant. However, micro-scale systems such as colloidal suspensions, microfluidic flows, and active fluids are frequently far from thermal equilibrium. In such environments, the heat bath itself may exhibit stochastic fluctuations that lead to time-varying thermal properties.

In this work, we allow the environmental properties, specifically the temperature in \eqref{exm}, to depend on a fast-varying stochastic process $y_t^\epsilon$. This process captures local fluid fluctuations arising from thermal conduction noise, optical trapping, random energy exchange, and the nonequilibrium behavior inherent in crowded fluids and active matter. Accordingly, $y_t^\epsilon$ represents rapidly varying environmental fluctuations and can be modeled by an Ornstein-Uhlenbeck process or, more generally, by a stochastic differential equation.

For fixed $x\in \rr^{d_1}$, the fast-varying random environment $y^\e$ in \eqref{exm} has a unique invariant measure $\mu^x(dy)$. Define $
\lbar{\TT^{-1}}(x): = \int_{\rr^{d_2}} 1/\TT(y) \mu^x(dy)$.
By \thmref{thm:SK-avg}, the limiting dynamics, as $\e \to 0$, becomes
\bea\ad 
dx_t = \bigg\{\frac{D(x_t) F(x_t)}{\kappa_B}\lbar{\TT^{-1}}(x_t) + \nabla_x D(x_t)\bigg\} dt + \sqrt{2 D(x_t)}dW_t.
\eea
The noise-induced drift $S(x)=\nabla_x D(x)$ coincides with the result in \cite[Example 4.1]{HMVW15}. When the temperature is constant $\TT\cd\equiv \TT_0$, the small-mass limit reduces to the classical limit:
\bea\ad
d\wdh x_t = \bigg\{\frac{D(\wdh x_t) F(\wdh x_t)}{\kappa_B \TT_0} + \nabla_x D(\wdh x_t)\bigg\} dt + \sqrt{2D(\wdh x_t)} dW_t.
\eea

\subsection{Numerical Experiments}\label{sec:num}
We present a numerical experiment for the system in \eqref{exm}. For simplicity, we set the Boltzmann constant $\kappa_B=1.0$ and the constant temperature $\TT_0 =1$. The coefficients are chosen as $
F(x) =-x, D(x)=1+0.05 \cos(2\pi x), \TT(y)=1.0+0.5\tanh(y)$, $
f(x,y) =-\alpha \tanh(0.5y)+ \beta \tanh(0.5x)$.
Thus, $\nabla_x D(x)=-0.1 \pi \sin(2 \pi x)$. We set parameters $\al=\beta=\sg=1.0$, terminal time $T=1.0$, initial values $x_0^m=v_0^m=y_0^\e=x_0=\wdh{x}_0=0$. We take Euler-Maruyama scheme for the simulation of stochastic differential equations. For this choice of $f(x,y)$, for any fixed $x\in \rr^{d_1}$, the fast-varying process with frozen slow variable has a unique Gibbsian invariant measure $\mu^x(dy)= \rho^x(y)dy$, where
\bea\ad
\rho^x(dy)= \frac{1}{\ZZ(x)}\big[\cosh(0.5y)\big]^{-\frac{4\alpha}{\sg^2}} \exp \bigg\{\frac{2\beta\tanh(0.5x) y}{\sg^2} \bigg\},
\eea 
and $\ZZ(x)$ is the normalization constant given by 
\bea 
\ZZ(x)\ad =\int_\rr \big[\cosh(0.5y)\big]^{-\frac{4\alpha}{\sg^2}} \exp\bigg\{\frac{2\beta\tanh(0.5x)}{\sg^2}y\bigg\} dy \\
\ad = 2^{\frac{4\al}{\sg^2}}\, \text{Beta}\bigg(\frac{2\al+2\beta \tanh(0.5x)}{\sg^2}, \frac{2\al-2\beta\tanh(0.5x)}{\sg^2}\bigg).
\eea
Here $\text{Beta}(\cdot,\cdot)$ denotes the Beta function.

The \figref{fig:SK-avg-path} shows the sample path of $x_t^m, x_t$, and $\wdh x_t$ for a fixed realization of the Brownian motion with time scale parameter $\e=0.02$. The \figref{fig:SK-avg-yT} shows the sample path of fast-varying process $y_t^\e$ and the corresponding temperature $\TT(y_t^\e)$. We can see that as $\e$ gets smaller, the sample paths of $x_t^m$ and $x_t$ get closer to those of $x_t^m$ and $\wdh x$, respectively, which is consistent with our theoretical results. For $\e \in \{0.1, 0.05, 0.02, 0.01, 0.005, 0.002,0.001\}$, we compute the mean squared error $\EE|x-x^m|^2$ and $\EE|\wdh x-x^m|^2$ using $N_{\omega}=1000$ Monte Carlo samples. Here $N$ denote the number of time discretization points on $[0,T]$, and  $dt$ represents the corresponding time step. The numerical results are summarized in Table~\ref{tab:error_data} and \figref{fig:error-plot}.

\begin{figure}[ht]
\centering
\includegraphics[width=0.68\textwidth]{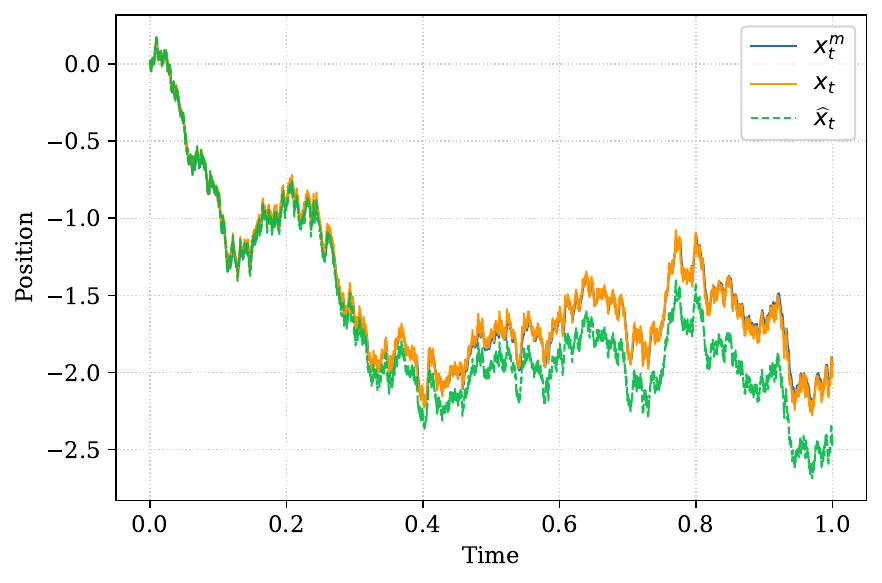}
\caption{Sample paths of $x_t^m, x_t$, and $\wdh x_t$ with $\e=0.02$.}\label{fig:SK-avg-path}

\end{figure}

\begin{figure}[ht]
\centering
\includegraphics[width=0.48\textwidth]{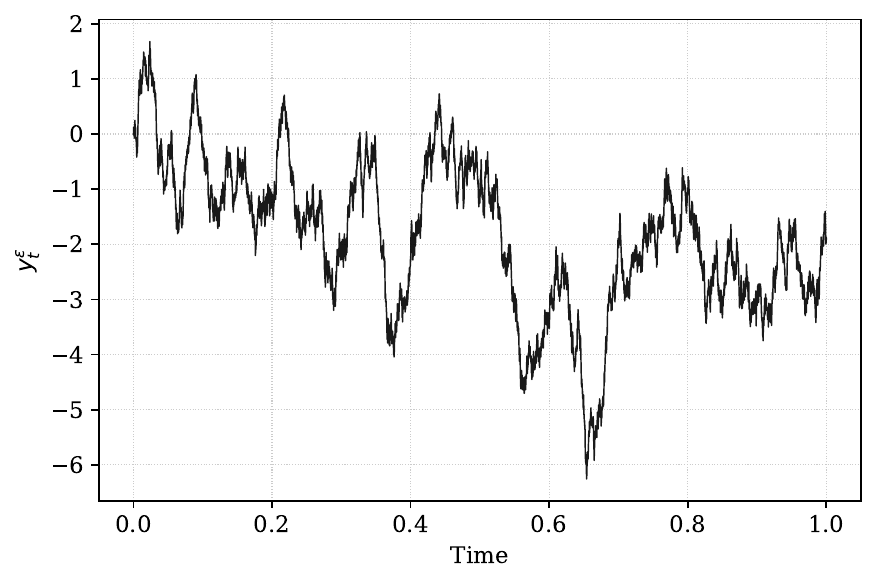}
\includegraphics[width=0.48\textwidth]{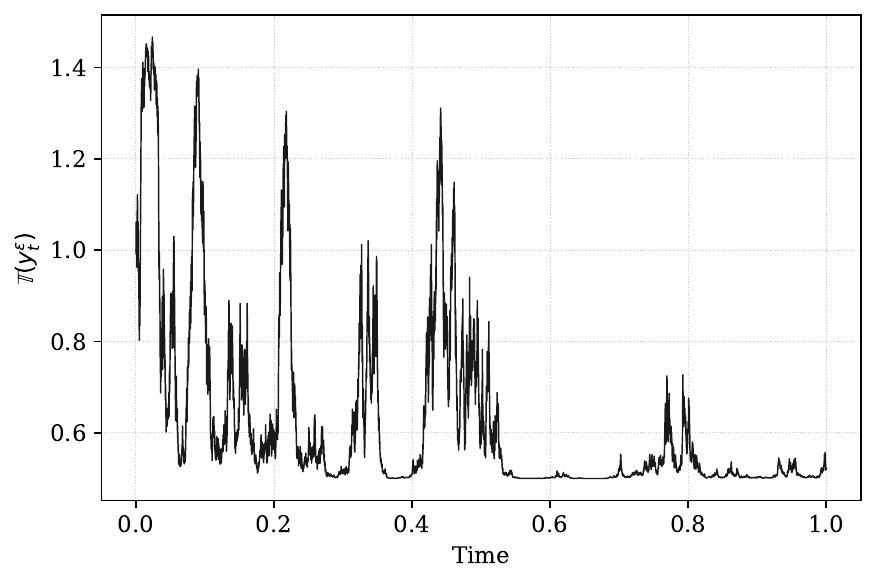}
\caption{Sample paths of $y_t^\e$ (left) and corresponding temperature $\TT(y_t^\e)$ (right) with $\e=0.02$.}
\label{fig:SK-avg-yT}
\end{figure}

\begin{table}[H]
    \centering    
    \caption{Mean Squared Error of $x-x^m$ and $\wdh x-x^m$ for various $\e$.}
    \label{tab:error_data}
    \begin{tabular}{c|c|c|c|c|}
        $\e$ & $N$ &  $dt$ &  $\EE|x-x^m|^2$ & $\EE|\wdh{x}-x^m|^2$ \\ 
\hline 
        0.100 & 10000   & 1.000000e-04 & 0.017933 & 0.017800 \\
        0.050  & 10000   & 1.000000e-04 & 0.006336 & 0.008132 \\
        0.020  & 10000   & 1.000000e-04 & 0.002021 & 0.005436 \\
        0.010 & 20000   & 5.000000e-05 & 0.000984 & 0.005185 \\
        0.005 & 80000   & 1.250000e-05 & 0.000567 & 0.005042\\
        0.002 & 500000  & 2.000000e-06 & 0.000292 & 0.005304 \\
        0.001 & 2000000 & 5.000000e-07 & 0.000189 & 0.005361
    \end{tabular}

\end{table}

\begin{figure}[H]
\centering
\includegraphics[width=0.68\textwidth]{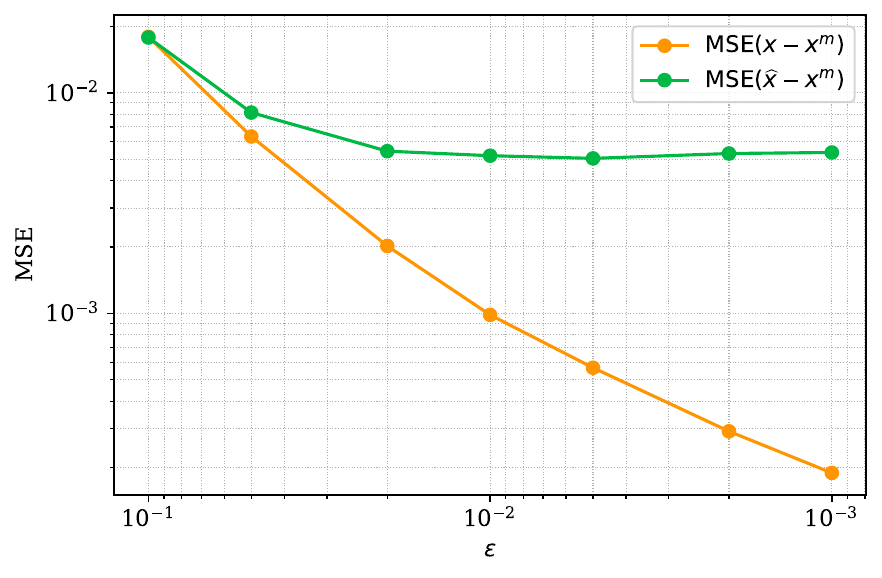}
\caption{Mean squared error log-log plot of $x-x^m$ and $\wdh x-x^m$ for $\e=0.001,0.002,0.005,0.01,0.02,0.05,0.1$.}
\label{fig:error-plot}
\end{figure}

\section{Concluding Remarks and Future Work}\label{sec:rem}
This work establishes the Smoluchowski--Kramers approximation for second-order stochastic differential equations with state-dependent friction in a nonequilibrium random environment. The environment is modeled by fast-varying stochastic differential equations, leading to a fully coupled fast--slow second-order system.

A natural direction for future research is to consider more intricate multiscale frameworks in which both the second-order system and the random environment involve additional rapidly oscillating coefficients, giving rise to homogenization effects; see \cite{RX21,RX21a}. A key question is to understand the behavior of the small-mass limit when environmental fluctuations and homogenization phenomena interact simultaneously. It is also of interest to identify the appropriate scaling under which central limit theorems hold, and to determine whether large and moderate deviation principles can be established in such regimes.

Another direction is to investigate strong convergence for the Smoluchowski--Kramers approximation in random environments, as in \cite{HLL26}. Establishing  convergence rates would provide a more refined quantitative understanding of the approximation and complement the asymptotic results developed in this work.

\appendix

\section{A Generalized Gr\"{o}nwall Inequality}\label{app:gronwall}

We recall the following generalized \gronwall inequality in \cite{CM67}.
\begin{thm}\label{thm:g-Gron}
Let the function $u$ and $f$ be continuous on the interval $[0,1]$ and let the function $\CK$ be continuous and non-negative on the triangle $0\leq s \leq t \leq 1$. If 
\beq{eq-uf}
u(t)\leq f(t)+ \int_0^t \CK(t,s) u(s)ds, \quad 0\leq t\leq 1,
\eeq
then 
\beq{gron-uf}
u(t) \leq f(t)+ \int_0^t \CH(t,s)f(s)ds, \quad 0\leq t \leq 1,
\eeq
where 
$\CH(t,s)=\sum_{n=1}^\infty \CK^{*,n}(t,s), 0\leq s \leq t \leq 1$, is the resolvent kernel, and the $\CK_n\,(n=1,2,\dots)$ are the iterated kernels of $\CK$.
\end{thm}
\begin{proof}
From \eqref{eq-uf}, one has 
\bea
u(t)\ad \leq f(t)+\int_0^t \CK(t,s) f(s)ds + \int_0^t \CK(t,s)\int_0^s \CK(s,r) u(r)dr ds \\
\aad = f(t)+ \int_0^t \CK_1(t,s)f(s) ds + \int_0^t \CK_2(t,s)u(s)ds,
\eea
for $0\leq t \leq 1$. The remainder follows by induction and a standard estimation procedure showing the resulting series to be uniformly convergence.
\end{proof}
\begin{lem}\label{lem:CH}
For $\CK:[0,T]\to [0,\infty)$ be locally integrable and define convolution powers
\bea\ad 
\CK^{*,1}=\CK,\quad \CK^{*,n+1}=\int_0^t \CK(t-s)\CK^{*,n}(s)ds
\eea
Then the following classical estimate for resolvent kernel holds
\bea\ad
\sum_{n=1}^\infty \CK^{*,n}(t) \leq \CK(t)\exp\bigg(\int_0^t \CK(s)ds \bigg).
\eea
\end{lem}
\begin{proof}
We will prove it by induction. For every $n \geq 1$, we claim that 
\beq{est-Cki}
\CK^{*,n}(t)\leq \frac{\CK(t)}{(n-1)!} \bigg(\int_0^t \CK(s)ds \bigg)^{n-1}.
\eeq
By induction, for $n=1$, \eqref{est-Cki} holds. Suppose \eqref{est-Cki} holds for $i\leq n$, then 
\bea
\CK^{*,n+1}(t)
\ad = \int_0^t \CK(t-s) \CK^{*,n}(s)ds \leq \int_0^t \CK(t-s) \frac{\CK(s)}{(n-1)!} \bigg(\int_0^s \CK(r)dr \bigg)^{n-1} ds
\eea
By monotonicity of the integral, we have 
\bea\ad 
\bigg(\int_0^s \CK(r)dr \bigg)^{n-1} \leq \bigg(\int_0^t \CK(r)dr \bigg)^{n-1}.
\eea
Thus,
\bea
\CK^{*,n+1}(t)\ad \leq \frac{1}{(n-1)!}\bigg(\int_0^t \CK(r)dr \bigg)^{n-1} \int_0^t \CK(t-s)\CK(s)ds \\
\aad \leq \frac{1}{n!} \bigg(\int_0^t \CK(r)dr \bigg)^{n-1}\CK(t) \int_0^t \CK(s)ds
\eea
where the last inequality follows from \eqref{est-Cki} for $n=2$. Therefore, the proof is complete.
\end{proof}

\section{Proofs of Auxiliary Lemmas}\label{app:lem-proof}

\subsection{Proof of \lemref{lem:U-tau}}\label{app:lem55}
\begin{proof}
Let us first prove \eqref{mv-W}. By the BDG inequality in \cite{KS91}, Fubini's theorem, and \eqref{MRn-bdd} in \lemref{lem:est-M}, we obtain for sufficiently small $\e$
\beq{mv-w-2}\barray\ad 
\EE \bigg( \sup_{t\in [0,T]} \Big|\int_0^t \big[M^{-1}_{R,n} \big]_{ij}(x_s^m, y_s^\e) e_1(x_s^m, y_s^\e) m (v_s^m)_i d(W_s)_j \Big| \bigg)^2  \\
\aad \leq C \int_0^T \EE \big|\big[M^{-1}_{R,n} \big]_{ij}(x_s^m, y_s^\e) e_1(x_s^m, y_s^\e) m(v_s^m)_i \big|^2 ds \\ 
\aad \leq 2 \, C_{e_1} \int_0^T \EE \big(1+|x_s^m|^2 \big) |m (v_s^m)_i|^2 ds.
\earray\eeq

Let $K>0$ be any positive constant. Define $\Omega_1:=\{(1+|x_s^m|^2) \leq K\}$ and denote by $\Omega_1^c$ its complement. By the H\"{o}lder inequality and \eqref{0-mvs}, \eqref{4-mv} in \lemref{lem:mv}, we have
\beq{mv2-omg}\barray\ad
\EE\big\{ \big(1+|x_s^m|^2 \big) |m (v_s^m)_i|^2\big\} \\
\aad =\EE\Big\{ \big(1+|x_s^m|^2 \big) \indi_{\Omega_1} |m(v_s^m)_i|^2 \Big\} + \EE \Big\{ \big(1+|x_s^m|^2 \big) \indi_{\Omega_1^c} |m (v_s^m)_i|^2 \Big\} \\
\aad \leq K\, \EE|m(v_s^m)_i|^2 + \Big\{ \EE \Big[ \Big(1+|x_s^m|^2 \Big)^2 \indi_{\Omega_1^c}\Big] \Big\}^{1/2} \Big(\EE|m (v_s^m)_i|^4\Big)^{1/2} \\
\aad \leq K\,C_{T,\la_0,v_0} m + \sqrt{C_{T,\la_0,v_0}} \bigg\{ \EE \bigg[ \bigg(1+\sup_{s\in [0,T]} |x_s^m|^2 \bigg)^2 \indi_{\Omega_1^c} \bigg] \bigg\}^{1/2}.
\earray\eeq
The H\"{o}lder inequality and the Chebyshev inequality further yield that
\beq{omg-c}\barray\ad
\EE\bigg[ \bigg(1+\sup_{s\in [0,T]}|x_s^m|^2 
\bigg)^2 \indi_{\Omega_1^c}\bigg] \\
\aad \leq \bigg\{\EE\bigg(1+\sup_{s\in [0,T]}|x_s^m|^2 \bigg)^4 \bigg\}^{1/2} \bigg\{ \PP \bigg(1+|x_s^m|^2 > K\bigg) \bigg\}^{1/2} \\
\aad \leq \frac{C}{K} \bigg\{1+\EE \sup_{s\in [0,T]}|x_s^m|^8 \bigg\}^{1/2} \bigg\{1+\EE\sup_{s\in [0,T]}|x_s^m|^4 \bigg\}^{1/2} \leq \frac{C_{T,\la_0,v_0}}{K},
\earray\eeq
where the last line follows from \assmref{ass:mom-bdd} and $C_{T,\la_0,v_0}$ is a constant independent of $m$ and $\e$. Thus, combining \eqref{omg-c} and \eqref{mv2-omg}, letting $\e \to 0$ and then $K\to \infty$, we have
\bea\ad
\int_0^T \EE(1+|x_s^m|^2)|m(v_s^m)_i|^2 ds \leq C_{e_1} T \Big(K C_{T,\la_0,v_0}m+\sqrt{C_{T,\la_0,v_0}}/\sqrt{K} \Big) \to 0, 
\eea
which implies \eqref{mv-W}.

For \eqref{inner-mv}, 
the linear growth of $e_1(x,y)$ and \eqref{MRn-bdd} imply that for sufficiently small $\e>0$,
\bea\ad
\big|\big[M^{-1}_{R,n} \big]_{ij}(x_s^m, y_s^\e) e_1(x_s^m, y_s^\e) m (v_s^m)_i \big| \leq 2\, C_{e_1} (1+|x_s^m|)|m(v_s^m)_i|.
\eea 
The Cauchy-Schwarz inequality and the Fubini's theorem yield that
\beq{M-e1}\barray\ad 
\EE \bigg[\sup_{t\in [0,T]} \bigg|\int_0^t \big[M^{-1}_{R,n} \big]_{ij}(x_s^m, y_s^\e) e_1(x_s^m, y_s^\e) m (v_s^m)_i ds \bigg|^2\bigg]\\
\aad \leq  C_{T,e_1} \int_0^T  \EE(1+|x_s^m|)^2|m(v_s^m)_i|^2 ds \to 0, \quad \text{as } m\to 0,
\earray\eeq
where the last line follows from the argument of \eqref{mv-w-2}.

In the light of \eqref{ev-B}, the BDG inequality implies
\bea\ad 
\EE \bigg[\bigg(\sup_{t\in [0,T]} \bigg| \sqe \int_0^t \big[M^{-1}_{R,n} \big]_{ij}(x_s^m, y_s^\e) e_2(x_s^m, y_s^\e) \sqrt{m} (v_s^m)_\iota d(B_s)_\ell \bigg|\bigg)^2 \bigg] \\
\aad \leq \e \int_0^T \EE  \big|[M_{R,n}^{-1}]_{ij}(x_s^m, y_s^\e) e_2(x_s^m, y_e^\e) \sqrt{m} (v_s^m)_\iota \big|^2 ds \\
\aad \leq 2 C_{e_2} \e \int_0^T \EE \big|\sqrt{m} (v_s^m)_{\iota} \big|^2 ds \leq 2 T C_{e_2} C_{T,\la_0, |v_0|} \e \to 0.
\eea
Therefore, the proof of this lemma is complete.
\end{proof}

\subsection{Proof of \lemref{lem:U-dl}}\label{app:lem56}
\begin{proof}
For \eqref{M-I-ds}, \lemref{lem:est-M} implies 
\bea\ad
\EE\bigg|\sup_{t\in [0,T]} \int_0^t \Big([M_{R,n}^{-1}]_{ij}(x_s^m, y_s^\e)(x_s^m, y_s^\e) -\dl_{ij} \Big) e_1(x_s^m, y_s^\e) ds \bigg|^2 \\
\aad \leq C_T\, \EE \int_0^T \big|[M_{R,n}^{-1}]_{ij}(x_s^m, y_s^\e)-\dl_{ij} \big|^2 |e_1(x_s^m, y_s^\e)|^2 ds \\
\aad \leq C_{T,e_1} \frac{(\sqrt{m}C_R)^2 }{ (1-\sqrt{m}C_R)^2} \EE \int_0^T 1+|x_s^m|^2 ds \\
\aad \leq C_{T,e_1}  \frac{(\sqrt{m}C_R)^2 }{ (1-\sqrt{m}C_R)^2} \EE \bigg(1+\sup_{s\in [0,T]}|x_s^m|^2 \bigg) \to 0 \quad \text{ as } \e  \to 0.
\eea
Similarly, the BDG inequality implies
\bea\ad
\EE\bigg|\sup_{t\in [0,T]} \int_0^t \Big( [M_{R,n}^{-1}]_{ij}(x_s^m, y_s^\e)-\dl_{ij} \Big) e_1(x_s^m, y_s^\e) dW_s \bigg|^2 \\
\aad \leq \EE\int_0^T \big|[M_{R,n}^{-1}]_{ij}(x_s^m, y_s^\e)-\dl_{ij}\big|^2 |e_1(x_s^m, y_s^\e)|^2 ds \to 0, \quad \text{ as } \e \to 0.
\eea
Consequently, we complete the proof of this lemma.
\end{proof}

\subsection{Proof of \lemref{lem:U-phi-psi}}\label{app:lem57}
\begin{proof}
In terms of \eqref{e3n-B}, by the BDG inequality, statement (i) in \thmref{thm:Poi}, and the uniform boundedness of $g$, we obtain that as $\e\to 0$,
\bea\ad 
\EE\bigg|\sup_{t\in [0,T]} \sqe \int_0^t \big[M^{-1}_{R,n} \big]_{ij}(x_s^m, y_s^\e)\frac{\partial}{\partial y_\ell} \phi_{i}^{R,n}(x_s^m, y_s^\e) \big[g(x_s^m, y_s^\e)dB_s \big]_\ell \bigg|^2 \\
\aad \leq \e \, \EE \int_0^T  \Big|\big[M^{-1}_{R,n} \big]_{ij}(x_s^m, y_s^\e) \frac{\partial}{\partial y_\ell} \phi_{i}^{R,n}(x_s^m, y_s^\e)g_{\ell \iota}(x_s^m, y_s^\e) \Big|^2 ds  \\
\aad \leq  4 C_0^2(R) C_T\, \e \to 0.
\eea
For \eqref{d-e3n}, we split it into two terms
\bea\ad 
\e \int_0^t [M_{R,n}^{-1}]_{ij}(x_s^m, y_s^\e) d\big[\phi_{i}^{R,n}(x_s^m, y_s^\e)\big] \\
\aad = \e \int_0^t \Big([M_{R,n}^{-1}]_{ij}(x_s^m, y_s^\e) -\dl_{ij}\Big) d\big[\phi_{i}^{R,n}(x_s^m, y_s^\e)\big]+ \e \int_0^t d\big[\phi_{j}^{R,n}(x_s^m, y_s^\e)\big] \\
\aad =: \CE_1^\e(t)+ \CE_2^\e(t).
\eea
By the boundedness of $\phi_{j}^{R,n}$ in \thmref{thm:Poi}, for each fixed $R>0$, we have 
\bea\disp
\lim_{\e \to 0} \EE \bigg| \sup_{t\in [0,T]} \CE_2^\e(t) \bigg|^2  \ad  \leq \lim_{\e \to 0}\EE \Big|\sup_{t\in [0,T]}\e \Big(\phi_{j}^{R,n}(x_t^m, y_t^\e)-\phi_{j}^{R,n}(x_0^m, y_0^\e)\Big)\Big|^2 \\
\aad \leq \lim_{\e \to 0} 4 C_0^2(R)\,\e^2  = 0
\eea
To deal with $\CE_1^\e$, we apply the It\^{o} formula to $\phi_{i}^{R,n}(x_s^m, y_s^\e)$ and obtain
\bea
\CE_1^\e(t)\ad =\e
\int_0^t \Big([M_{R,n}^{-1}]_{ij}(x_s^m, y_s^\e) -\dl_{ij}\Big)
\nabla_x \phi_{i}^{R,n}(x_s^m, y_s^\e) v_s^m ds \\
\aad\quad + \e \int_0^t  \Big([M_{R,n}^{-1}]_{ij}(x_s^m, y_s^\e) -\dl_{ij}\Big) \nabla_y \phi_{i}^{R,n}(x_s^m, y_s^\e) \cdot \frac{1}{\e} f(x_s^m,y_s^\e) ds \\
\aad \quad + \e \int_0^t \Big([M_{R,n}^{-1}]_{ij}(x_s^m, y_s^\e) -\dl_{ij}\Big)  \frac{\text{Tr}(\nabla_y^2 \phi_{i}^{R,n}(x_s^m, y_s^\e) (gg^\top)(x_s^m, y_s^\e))}{2\e} ds\\
\aad \quad + \e \int_0^t \Big([M_{R,n}^{-1}]_{ij}(x_s^m, y_s^\e) -\dl_{ij}\Big)  \nabla_y \phi_{i}^{R,n}(x_s^m, y_s^\e)\cdot \frac{1}{\sqe} g(x_s^m, y_s^\e)d B_s \\
\aad =: \CE_{1,1}^\e(t)+\CE_{1,2}^\e(t) + \CE_{1,3}^\e(t)+\CE_{1,4}^\e(t).
\eea
By the $L^2(\Omega)$ estimates for $\e v_s^m$ for each $s\in [0,T]$ in \eqref{bdd-mvs} and \eqref{M-I}, we obtain
\bea\ad 
\lim_{\e \to 0}\EE \bigg|\sup_{t\in [0,T]} \CE_{1,1}^\e(t)\bigg|^2 =0.
\eea
The detail is omitted here. Employing an analogous argument with \lemref{lem:U-dl} for $\CE_{1,2}^\e(t),\CE_{1,3}^\e(t)$, and $\CE_{1,4}^\e(t)$, we are able to prove they converge to zero as $\e\to 0$ in $L^2(\Omega;C([0,T];\rr))$.

Finally, as for \eqref{e3n-CL}, the definition of $\CL_y^x$ in \eqref{CL-x} and assumptions (B1)-(B2) yield that
\bea\ad 
\EE\bigg| \sup_{t\in [0,T]} \int_0^t \CL_y^{x_s^m} e_3^R(x_s^m, y_s^\e) -\CL_y^{x_s^m} e_3^{R,n}(x_s^m, y_s^\e)ds \bigg|^2  \\
\aad \leq C \EE \bigg( \int_0^T  \sum_{\ell=1,2} |(\nabla_y^\ell e_{3,i}^R-\nabla_y^\ell \phi_{i}^{R,n})(x_s^m,\cdot)|_\infty^2 ds \bigg) \leq C_0(R) n^{-2} \to  0, \quad \text{ as } n \to \infty,
\eea
for any fixed $R>0$,  where the last line follows from \lemref{lem:mol-est}. 

Applying an argument analogous to the one used for $\phi_i^{R,n}$ to the case of $\psi_i^n$ yields similar convergence results. This completes the proof.
\end{proof}

\textbf{Acknowledgements.} The author would like to thank Professor Longjie Xie for bringing the manuscript \cite{HLL26} to our attention.
\smallskip

\textbf{Funding.}
This research did not receive any specific grant from funding agencies in the public, commercial, or not-for-profit sectors.
\medskip

\textbf{Declaration of competing interest.}
The author declares that there are no conflicts of interest regarding the publication of this paper.


\end{document}